\documentclass[reqno]{amsart}

\usepackage{amsmath}
\usepackage{mathtools}
\usepackage{amssymb}
\usepackage{amsthm}
\usepackage{bbm}
\usepackage{dsfont}
\usepackage{mathrsfs}
\usepackage{commath}
\usepackage{youngtab}
\usepackage{subcaption}
\usepackage{graphicx}
\usepackage{pdfpages}

\DeclareMathOperator{\sgn}{sgn}
\DeclareMathOperator{\Var}{Var}
\DeclareRobustCommand{\stirling}{\genfrac\{\}{0pt}{}}

\newtheorem{theorem}{Theorem}[section]
\newtheorem{lemma}{Lemma}[section]
\newtheorem{proposition}{Proposition}[section]
\newtheorem{corollary}{Corollary}[section]

\theoremstyle{definition}
\newtheorem{definition}{Definition}[section]

\theoremstyle{remark}
\newtheorem{remark}{Remark}[section]

\numberwithin{equation}{section}

\usepackage[margin=1.35in]{geometry}
\begin{document}

\title[Eigenvalue Fluctuations of Symmetric Group Representations]{Eigenvalue Fluctuations of Symmetric Group Permutation Representations on k-tuples and k-subsets}

\author{Benjamin Tsou}
\address{Department of Mathematics, University of California, Berkeley, CA 94720-3840}
\email{benjamintsou@gmail.com}

\begin{abstract}
Let the term $k$-representation refer to the permutation representations of the symmetric group $\mathfrak{S}_n$ on $k$-tuples and $k$-subsets as well as the $S^{(n-k,1^k)}$ irreducible representation of $\mathfrak{S}_n$.  Endow $\mathfrak{S}_n$ with the Ewens distribution and let $\alpha$ and  $\beta$ be linearly independent irrational numbers over $\mathbb{Q}$.  Then for fixed $k > 1$ we show that as $n \to \infty$, the normalized count of the number of eigenangles in a fixed interval $(\alpha, \beta)$ of a $k$-representation evaluated at a random element $\sigma \in \mathfrak{S}_n$ converges weakly to a compactly supported distribution.  In particular, we compute the limiting moments and moreover provide an explicit formula for the limiting density when $k = 2$ and the Ewens parameter $\theta = 1$ (uniform probability measure).  This is in contrast to the $k = 1$ case where it has been shown previously that the distribution is asymptotically Gaussian.

\end{abstract}

\maketitle

\allowdisplaybreaks

\section{Introduction} \label{intro}

The group of permutation matrices can be viewed as the simplest (nontrivial) permutation representation of the symmetric group $\mathfrak{S}_n$.  Wieand \cite{wieand} showed that under a uniform probability measure, the normalized limiting distribution of the number of eigenvalues of a random permutation matrix in some fixed arc of the unit circle follows a standard normal distribution.  Recently, Ben Arous and Dang \cite{arousdang} have extended Wieand's work in \cite{wieand} to general functions other than the indicator function on an interval.  In particular, they show that the fluctuations of sufficiently smooth linear statistics of permutation matrices drawn from the Ewens distribution are asymptotically non-Gaussian but infinitely divisible.  They mention that this result is quite unusual since most prior work show asymptotic Gaussianity of eigenvalue fluctuation statistics.  

In this paper, we extend Wieand's results in a different direction by studying higher dimensional representations of the symmetric group.  In particular, we will consider the permutation representations on ordered $k$-tuples and unordered subsets of size $k$ as well as the irreducible representation $S^{(n-k,1^k)}$ for $k \ge 2$.  We show that for these three types of representations (denoted $\rho_{n,k}^\textrm{tuple}$, $\rho_{n,k}^\textrm{set}$, and $\rho^{(n-k,1^k)}$), for $\sigma \in \mathfrak{S}_n$ drawn from the Ewens distribution, the normalized count of eigenvalues in some fixed arc of the unit circle converges to a class of compactly supported limiting distributions.

Let us now quickly review how permutation representations of $k$-tuples and $k$-subsets are defined.  Section \ref{irrep} will give a short overview of irreducible representations of symmetric groups.  

First, consider the set $Q_{n,k}^\textrm{tuple}$ of ordered $k$-tuples $(t_1,...,t_k)$ of distinct integers chosen from the set $[n] := \{1,...,n\}$.  The symmetric group $\mathfrak{S}_n$ acts naturally on this set by $\sigma(t_1,...,t_k) = (\sigma(t_1),...,\sigma(t_k))$.  We can form the $\displaystyle{\frac{n!}{(n-k)!}}$-dimensional vector space $V_{n,k}^{\textrm{tuple}}$ with basis elements $e_{(t_1,...,t_k)}$.  Then the action of $\mathfrak{S}_n$ on $Q_{n,k}^\textrm{tuple}$ induces the permutation representation $\rho_{n,k}^{\textrm{tuple}}: \mathfrak{S}_n \rightarrow O(V_{n,k}^{\textrm{tuple}})$ where $O(V_{n,k}^{\textrm{tuple}})$ is the orthogonal group on $V_{n,k}^{\textrm{tuple}}$.

Similarly, consider the set $Q_{n,k}^\textrm{set}$ of $k$-subsets $\{t_1,...,t_k\}$ of distinct integers chosen from $[n]$.  As for the set of ordered tuples, the symmetric group $\mathfrak{S}_n$ acts naturally on $Q_{n,k}^\textrm{set}$ by $\sigma(\{t_1,...,t_k\}) = \{\sigma(t_1),...,\sigma(t_k)\}$.  We can form the $\displaystyle{ \binom{n}{k} }$-dimensional vector space $V_{n,k}^{\textrm{set}}$ with basis elements $e_{\{t_1,...,t_k\}}$.  Then the action of $\mathfrak{S}_n$ on $Q_{n,k}^\textrm{set}$ gives the permutation representation $\rho_{n,k}^{\textrm{set}}: \mathfrak{S}_n \to O(V_{n,k}^{\textrm{set}})$.  

To state the main results, let us introduce the relevant random variables describing the eigenvalue statistics of these symmetric group representations.  Finite group representations are all unitarisable, and therefore all the eigenvalues are of the form $e^{2 \pi i \phi}$ on the unit circle.  It will be convenient to refer to each eigenvalue $e^{2 \pi i \phi}$ by its \textit{eigenangle} $\phi \in [0, 1)$.  Let $I = (\alpha, \beta)$ be an interval where $\alpha$ and $\beta$ are irrational and linearly independent over $\mathbb{Q}$.  For $\sigma \in \mathfrak{S}_n$, let $X_{n,k}^{\textrm{tuple}}(\sigma)$, $X_{n,k}^{\textrm{set}}(\sigma)$, and $X_{n,k}^{\textrm{irrep}}(\sigma)$ denote the number of eigenangles (counted with multiplicity) of $\rho_{n,k}^\textrm{tuple}(\sigma)$, $\rho_{n,k}^\textrm{set}(\sigma)$, and $\rho^{(n-k,1^k)}(\sigma)$ respectively in the arc $I$.  

Recall (see e.g. \cite{ewens}) that the Ewens distribution with parameter $\theta > 0$ defined on $\mathfrak{S}_n$ is given by 
\begin{equation}
\mathbb{P}(\sigma) = \frac{\theta^{K(\sigma)}}{\theta (\theta+1) \cdot \cdot \cdot (\theta + n - 1)}
\end{equation}
where $K(\sigma)$ is the total number of cycles of the permutation $\sigma$.  By equipping $\mathfrak{S}_n$ with the Ewens measure, we can think of $X_{n,k}^{\textrm{tuple}}$, $X_{n,k}^{\textrm{set}}$, and $X_{n,k}^{\textrm{irrep}}$ as random variables.  

For $k>1$, define the centered and scaled versions
\begin{equation}
Y_{n,k}^{\textrm{tuple}} := \frac{X_{n,k}^{\textrm{tuple}} - \mathbb{E}[X_{n,k}^{\textrm{tuple}}]}{n^{k-1}}
\end{equation} 
\begin{equation}
Y_{n,k}^{\textrm{set}} := k! \frac{X_{n,k}^{\textrm{set}} - \mathbb{E}[X_{n,k}^{\textrm{set}}]}{n^{k-1}}
\end{equation} 
\begin{equation}
Y_{n,k}^{\textrm{irrep}} := k! \frac{X_{n,k}^{\textrm{irrep}} - \mathbb{E}[X_{n,k}^{\textrm{irrep}}]}{n^{k-1}}
\end{equation} 
Our first result is to show that to compute the limiting distribution of these normalized eigenangle counts, it suffices to consider the simpler random variables $Y_{n,k}$ defined below.  For each $\sigma \in \mathfrak{S}_n$, let $C_j^{(n)}(\sigma)$ denote the number of cycles of $\sigma$ of length $j$.  Let $\mathcal{L}(X)$ denote the law of a random variable $X$. 

\begin{theorem} \label{simplifythm}
Let \begin{equation}
Y_{n,k} = \sum_{j=1}^n \frac{j^{k-1} C_j^{(n)}  (\{j \alpha\} - \{j \beta\})}{n^{k-1}}
\end{equation} 
Under the Ewens distribution with parameter $\theta > 0$, as $n \to \infty$ for fixed $k > 1$, each of the random variables $Y_{n,k}^{\textrm{tuple}}$, $Y_{n,k}^{\textrm{set}}$, and $Y_{n,k}^{\textrm{irrep}}$ converges in law to the same limiting distribution $\displaystyle{\lim_{ n\to \infty} \mathcal{L} (Y_{n,k})}$ (assuming it exists).
\end{theorem}

\begin{remark}
In fact, the proof of Theorem \ref{simplifythm} readily shows that a similar result stated in Theorem \ref{simplifythmgeneral} holds for more general linear eigenvalue statistics than the indicator of an interval.  
\end{remark}

Since $\displaystyle{\sum_{j=1}^n j C_j^{(n)} = n}$, it is easy to see that $\displaystyle{ \sum_{j=1}^n j^{k-1} C_j^{(n)} \le n^{k-1}}$. Thus, the distribution of $Y_{n,k}$ is supported on the finite interval $[-1,1]$ for all $n$.  Hence by the method of moments, the sequence converges in distribution as long as the moments converge.  The following theorem gives the limiting moments implicitly in terms of the exponential of a formal power series. 

\begin{theorem} \label{momentsthm}
Under the Ewens distribution with parameter $\theta > 0$, for $k > 1$, $Y_{n,k}$ converges weakly as $n \to \infty$ to some compactly supported limiting distribution $Y_{\infty, k}$.  The moments of $Y_{\infty, k}$ are given implicitly by the following equation in formal power series: 

\begin{equation}\label{power} \sum_{m = 0}^\infty \mathbb{E}[(Y_{\infty, k})^m] \frac{\Gamma(m(k-1)+\theta)}{\Gamma(\theta) m!} z^m = \exp(K(z))
\end{equation} where $\displaystyle{K(z) = \sum_{m=1}^\infty \kappa_{2m} z^{2m}}$ and $\displaystyle{\kappa_{2m} = \frac{2\theta(2m(k-1)-1)!}{(2m+2)!} }$. 

\end{theorem}

Remarkably, when $k=2$ and $\theta=1$, an explicit formula for the density of $Y_{\infty, 2}$ can be obtained.  

\begin{corollary} \label{twocor}
For $\theta=1$ (i.e. the uniform measure on $\mathfrak{S}_n$), the random variable $Y_{\infty,2}$ is supported on the interval $[-1,1]$ and has probability density given by the formula:

\begin{equation}\label{probdensity2} 
p_{Y_{\infty, 2}}(t) = \frac{e^{3/2}}{\pi |t|} \left(\frac{1}{|t|} - 1\right)^{-\frac{1}{2} (1 - |t|)^2} \left(\frac{1}{|t|} + 1\right)^{-\frac{1}{2} (1 + |t|)^2} \sin \bigg(\frac{(1-|t|)^2}{2} \pi \bigg) 
\end{equation}
for $-1 \le t \le 1$ \bigg(and by continuity, $\displaystyle{p_{Y_{\infty, 2}}(0) = \frac{e^{3/2}}{\pi}} \bigg)$. 
\end{corollary}

The rest of the paper is organized as follows.  In sections \ref{ktuple}, \ref{ksubset}, and \ref{irrep}, we prove Theorem \ref{simplifythm} in turn for $Y_{n,k}^{\textrm{tuple}}$, $Y_{n,k}^{\textrm{set}}$, and $Y_{n,k}^{\textrm{irrep}}$.  In section \ref{equidistributed sequences}, we review some basic theory of equidistributed sequences that will be useful for the moment method.  In section \ref{k=1}, we use the method of moments to rederive the asymptotic gaussianity of the normalized eigenangle count in the $k=1$ case of permutation matrices.  Then we move on to the $k>1$ case and prove Theorem \ref{momentsthm} in Section \ref{k>1}.  Section \ref{k=2} proves the density formula for $k=2$ in Corollary \ref{twocor}.  Finally in Section \ref{generalf}, we discuss the extension of Theorem \ref{simplifythm} to more general linear eigenvalue statistics and connect our results to those in \cite{arousdang}.

We end this introduction with a few bibliographic comments regarding the increasing activity in the study of eigenvalues of random permutation matrices over the last two decades.  Wieand extended her Gaussianity results in \cite{wieand} to wreath products in \cite{wieand2}.  Works by Diaconis and Shahshahani \cite{diaconisshah} and Evans \cite{evans} show that the spectrum of permutation matrices and various wreath products under a uniform probability measure converges weakly to the uniform distribution on the circle.  Characteristic polynomials associated to a random permutation matrix were studied in several works, including \cite{cook}, \cite{bahierA}, \cite{hambly}, \cite{zeindler}, and \cite{dangzeindler}.  Najnudel and Nikeghbali \cite{najnudel} and more recently Bahier (\cite{bahierB}, \cite{bahierC}) extend the work of Diaconis, Evans, and Wieand by studying ``randomized'' permutation matrices where each matrix entry equal to 1 is replaced by i.i.d. variables taking values in $\mathbb{C}^*$.  The authors in \cite{hughesnajnudel} study a more general Ewens measure than the one considered by Ben Arous and Dang \cite{arousdang} and in this paper, also obtaining eigenvalue statistics fluctuation results.  Evans \cite{evans2} considers spectra of random matrices involving more general representations of the symmetric group $\mathfrak{S}_n$, but the situation is quite different from ours in that the randomness is not taken over $ \mathfrak{S}_n$.  

\section{The $k$-tuple representations} \label{ktuple}

In this section, we give a proof of Theorem \ref{simplifythm} for $Y_{n,k}^{\textrm{tuple}}$.  First, we give a simple characterization of the spectrum of $\rho_{n,k}^{\textrm{tuple}}(\sigma)$ for $\sigma \in \mathfrak{S}_n$.  Note that the eigenvalues only depend on the cycle structure of $\sigma$ since conjugacy classes in $\mathfrak{S}_n$ are determined by the cycle structure.

When $k=1$ (the defining representation of $\mathfrak{S}_n$), $\rho_{n,1}^{\textrm{tuple}}(\sigma) = M$ where $M$ is the permutation matrix corresponding to $\sigma$, i.e. $M_{ij} = 1$ if $j = \sigma(i)$ and 0 otherwise.  It is easy to see that each $j$-cycle in $\sigma$ corresponds to the set of $j$ eigenangles $\displaystyle{ \left\{0, \frac{1}{j},...,\frac{j-1}{j} \right\} }$.  For each $j$, we have $C_j^{(n)}(\sigma)$ copies of these eigenangles.

Then \begin{equation} X_{n,1}^{\textrm{tuple}}(\sigma) = n(\beta - \alpha) + \sum_{j=1}^n C_j^{(n)}(\sigma) (\{j \alpha\} - \{j \beta \}) \end{equation} where $\{x\}$ denotes the fractional part of $x$.  

More generally, $\rho_{n,k}^{\textrm{tuple}}(\sigma)$ is the $\displaystyle{\frac{n!}{(n-k)!} \times \frac{n!}{(n-k)!}}$ permutation matrix $M$ corresponding to the action of $\sigma$ on $V_{n,k}^{\textrm{tuple}}$.  Here, $M_{(t_1,...,t_k), (u_1,...,u_k)} = 1$ if $(u_1,...,u_k) = \sigma(t_1,...,t_k)$ and 0 otherwise.  Let $\sigma_k^{\textrm{tuple}}$ be the permutation of size $\displaystyle{\frac{n!}{(n-k)!} }$ corresponding to $M$.  Then looking at the cycle structure, we have \begin{equation} \label{Xnktuple} X_{n,k}^{\textrm{tuple}}(\sigma) = \frac{n!}{(n-k)!} (\beta - \alpha) + \sum_{j} C_{j,k}^{(n), \textrm{tuple}} (\sigma) (\{j \alpha\} - \{j \beta \}) \end{equation} where $C_{j,k}^{(n), \textrm{tuple}}(\sigma)$ is the number of cycles in $\sigma_k^{\textrm{tuple}}$ of length $j$.

We will say that an integer $i$ lies in cycle $C$ of the permutation $\sigma$ if $C$ contains $i$ in the cycle decomposition of $\sigma$.  It will turn out that almost all the contribution to the sum in $Y_{n,k}^{\textrm{tuple}}$ comes from the tuples ($t_1,...,t_k)$ such that $t_1,...,t_k$ all lie in the same cycle of $\sigma$.  
\begin{remark}
To reduce confusion, we will sometimes use the terms $\sigma$-cycle and $\sigma_k^{\textrm{tuple}}$-cycle to distinguish between cycles of $\sigma$ and $\sigma_k^{\textrm{tuple}}$ respectively.
\end{remark}
\begin{remark}
To reduce clutter, we will often leave the index off set and sequence notations.  For example, if the index $n$ is understood to run over the range $1 \le n \le N$, then the notation $(a_n)$ should be read as the sequence $(a_1,...,a_N)$.  Similarly, if the index $i$ is understood to run over the range $1 \le i \le m$, then $\{ A_i\}$ should be read as the set $\{A_1,...,A_m\}$.  
\end{remark}

In order to analyze the sum in $Y_{n,k}^{\textrm{tuple}}$, it will help to obtain a partition of the set of $k$-tuples $(t_1,...,t_k)$ defined by the orbits of the action of $\sigma \in \mathfrak{S}_n$.  First, we make the following: 
\begin{definition} 
Let $\displaystyle{[k] = \mathop{\cup}_{i = 1}^m A_i}$ be a partition of $[k]$ into disjoint, nonempty subsets.  Order the sets $A_i$ such that $|A_1| \ge...\ge |A_m| > 0$.  Setting $k_i = |A_i|$, this determines an integer partition $k = k_1+...+k_m$.  Further partition each subset $A_r$ into $p(r)$ disjoint, nonempty subsets $A_{rs}$ such that $|A_{r1}| \ge...\ge |A_{r, p(r)}| > 0$.  This determines an integer partition $\displaystyle{k_r = \sum_{s=1}^{p(r)} k_{rs}}$.  The pair $(\{A_i\}, \{A_{rs}\})$ will be called a double partition of $[k]$.  
\end{definition}

We can now define the following subsets of $Q_{n,k}^\textrm{tuple}$:

\begin{definition} \label{deftuples}
Let $(\{A_i\}, \{A_{rs}\})$ be a double partition of $[k]$ such that $|\{A_i\}| = m$.   Choose a sequence $(i_1,...,i_m)$ of distinct integers from $[n]$.  Then let $T_{\{A_i\}, \{A_{rs}\}}^{\sigma, (i_j)}$ denote the set of $k$-tuples $(t_1,...,t_k)$ of distinct integers such that the integers $t_{j}$ where $j \in A_r$ are all in $\sigma$-cycles of length $i_r$ and moreover, integers $t_{a}$ and $t_b$ are in the \textit{same} $\sigma$-cycle of length $i_r$ iff $a$ and $b$ are in the same subset $A_{rs}$ of $A_r$.  Taking the union over $\sigma$-cycle lengths, we also define \begin{equation}T_{\{A_i\}, \{A_{rs}\}}^{\sigma} = \bigcup_{i_1 \neq... \neq i_m} T_{\{A_i\}, \{A_{rs}\}}^{\sigma, (i_j)} \end{equation} where $i_1 \neq... \neq i_m$ is shorthand for $i_1,...,i_m$ all distinct.
\end{definition}
For each $\sigma \in \mathfrak{S}_n$, the set $ \Big\{T_{\{A_i\}, \{A_{rs}\}}^{\sigma}: (\{A_i\}, \{A_{rs}\}) \text{ a double partition of }[k]  \Big\}$ forms a partition of $Q_{n,k}^\textrm{tuple}$.  Note that the number of parts in this partition is only a function of $k$ and does not grow with $n$.  Thus, we can consider separately the limiting contribution of each part $T_{\{A_i\}, \{A_{rs}\}}^{\sigma} $ to the spectrum of $\rho_{n,k}^{\textrm{tuple}}(\sigma)$.

It is clear that $\sigma_k^{\textrm{tuple}}$ acts on $T_{\{A_i\}, \{A_{rs}\}}^{\sigma, (i_j)}$ and that each tuple $(t_1,...,t_k) \in T_{\{A_i\}, \{A_{rs}\}}^{\sigma, (i_j)}$ lies in a $\sigma_k^{\textrm{tuple}}$-cycle of length $[i_1,...,i_m]$.  (Here, $[i_1,...,i_m]$ denotes the least common multiple of integers $i_1,...,i_m$).  Thus, for each choice of double partition $(\{A_i\}, \{A_{rs}\})$ and each sequence of $\sigma$-cycle lengths $(i_1,...,i_m)$, the elements of $T_{\{A_i\}, \{A_{rs}\}}^{\sigma, (i_j)} $ form \begin{equation}\label{numform} \frac{1}{[i_1,...,i_m]} \prod_{r=1}^m  \big(C_{i_r}^{(n)}(\sigma) \big)^{\underline{p(r)}} \prod_{s=1}^{p(r)} i_r^{\underline{k_{rs}}} \end{equation}  $\sigma_k^{\textrm{tuple}}$-cycles of size $[i_1,...,i_m]$ where the polynomial $x^{\underline{n}} := x(x-1)...(x-n+1)$ (often called the $n$th falling factorial).

The following lemma shows that the only non-negligible contribution to $Y_{n,k}^{\textrm{tuple}}$ in the limit $n \to \infty$ comes from the $\sigma_k^{\textrm{tuple}}$-cycles containing tuples $(t_1,...,t_k)$ such that $t_1,...,t_k$ are all in the same $\sigma$-cycle.   

\begin{lemma}\label{tuplebound} Let $(\{A_i\}, \{A_{rs}\})$ be a double partition of $[k]$.  If $\displaystyle{\sum_{r=1}^m p(r) > 1}$, then \begin{equation} \lim_{n \to \infty} \frac{1}{n^{k-1}} \mathbb{E} \bigg[\sum_{i_1 \neq... \neq i_m} \frac{1}{[i_1,...,i_m]} \prod_{r=1}^m  \big(C_{i_r}^{(n)} \big)^{\underline{p(r)}} \prod_{s=1}^{p(r)} i_r^{\underline{k_{rs}}} \bigg] = 0 \end{equation}   
\end{lemma}  
  
\begin{proof}  
First, we compute the expectation $\displaystyle{\mathbb{E}\bigg[ \prod_{r=1}^m  \big(C_{i_r}^{(n)} \big)^{\underline{p(r)}} \bigg]}$, often referred to as a factorial moment.

The following moment formula was established by Watterson \cite{watterson} (see e.g. \cite[(5.6)]{a}): Let $W_1,...,W_n$ be independent Poisson random variables with $\mathbb{E}[W_i] = \theta/i$ and $b_1,...,b_n$ be non-negative integers and set $l = b_1+2b_2+...+n b_n$.  Then \begin{equation}\label{diaconispoisson} \mathbb{E} \bigg[\prod_{j=1}^n (C_j^{(n)})^{\underline{b_j}} \bigg] = \mathds{1}(l \le n) \mathbb{E} \bigg[\prod_{j=1}^n W_j^{\underline{b_j}} \bigg] \prod_{i=0}^{l-1} \frac{n-i}{\theta+n-i-1} 
\end{equation}

Also, recall that if $X$ follows a Poisson distribution with parameter $\lambda$, the factorial moments are given by $\mathbb{E}[X^{\underline{n}}] = \lambda^n$.

Using these results, we can now compute the expectation in the lemma.  Note that 
\[ \prod_{i=0}^{l-1} \frac{n-i}{\theta+n-i-1} \le \frac{n}{n-l+\theta} \]

Then summing over all sequences $(i_1,...,i_m)$ of distinct integers in $[n]$, we get (assuming $\displaystyle{\sum_{r=1}^m p(r) > 1}$)
\begin{align}
&\nonumber \mathbb{E} \bigg[\sum_{i_1 \neq... \neq i_m} \frac{1}{[i_1,...,i_m]} \prod_{r=1}^m  \big(C_{i_r}^{(n)} \big)^{\underline{p(r)}} \prod_{s=1}^{p(r)} i_r^{\underline{k_{rs}}} \bigg] \\
\le &\nonumber \sum_{i_1 \neq ...\neq i_m} \frac{1}{[i_1,...,i_m]}  \frac{n}{n-\sum i_r p(r) +\theta} \prod_{r=1}^m \left( \frac{\theta}{i_r} \right)^{p(r)} \prod_{s=1}^{p(r)} i_r^{k_{rs}} \\ 
\le &\label{sum2} A(\theta) \sum_{\substack{i_1 \neq ...\neq i_m \\ \sum i_r p(r) < n}} \frac{1}{[i_1,...,i_m]}  \frac{n}{n-\sum i_r p(r)} \prod_{r=1}^m i_r^{k_r - p(r)} 
\end{align}
for some constant $A(\theta)$.  

If $m = 1$, splitting the sum according to whether $n - i_1 p(1) > \sqrt{n}$ or $n - \sum i_1 p(1) \le \sqrt{n}$ shows that \eqref{sum2} is clearly of order $O(n^{k-3/2})$ and the lemma follows.  If $m>1$, we have (using the notation $(i,j):= \gcd(i,j)$ )

\begin{align*}
&\sum_{\substack{i_1 \neq ...\neq i_m \\ \sum i_r p(r) < n}} \frac{1}{[i_1,...,i_m]}  \frac{n}{n-\sum i_r p(r)} \prod_{r=1}^m i_r^{k_r - p(r)}  \\
\le & \sum_{\substack{i_1 \neq ...\neq i_m \\ \sum i_r p(r) < n}} \frac{1}{[i_1,i_2]}  \frac{n}{n-\sum i_r p(r)} \prod_{r=1}^m i_r^{k_r - p(r)}   \\
< & \; n^{k - k_1 - k_2+1/2} \sum_{i \neq j}\frac{i^{k_1 - 1} j^{k_2 - 1}}{[i,j]}\\
= & \; 2n^{k - k_1 - k_2+1/2} \sum_{1 \le i < j \le n} i^{k_1 - 2} j^{k_2 - 2} (i,j) \\
= & \; 2n^{k - k_1 - k_2+1/2} \sum_{\substack{1 \le i < j \le n \\ (i,j)=1}} \sum_{d \le \frac{n}{j}} (d i)^{k_1 - 2} (d j)^{k_2 - 2} d \\
< & \; 2n^{k - k_1 - k_2+1/2} \sum_{1 \le i < j \le n} \sum_{d \le \frac{n}{j}} d^{k_1+k_2-3} i^{k_1 - 2} j^{k_2 - 2} \\
< & \; 2n^{k - k_1 - k_2+1/2} \sum_{i=1}^n\sum_{j=i}^n \left(\frac{n}{j}\right)^{k_1+k_2-2} i^{k_1 - 2} j^{k_2 - 2} \\
< & \; 2n^{k - 3/2} \sum_{i=1}^n \sum_{j=i}^n \frac{i^{k_1 - 2}}{j^{k_1}}
\end{align*} 
Here, the second step is derived by splitting the sum according to whether $n - \sum i_r p(r) > \sqrt{n}$ or $n - \sum i_r p(r) \le \sqrt{n}$.  The desired result then follows from the fact that \[\sum_{i=1}^n \sum_{j=i}^n \frac{i^{k_1 - 2}}{j^{k_1}} = O(\log^2 n) \]
\end{proof}

Lemma \ref{tuplebound} shows that the only cycles of $\sigma_k^{\textrm{tuple}}$ that contribute to $Y_{n,k}^{\textrm{tuple}}$ in the limit $n \to \infty$ are those formed from tuples in the set $T_{\{A_i\}, \{A_{rs}\}}^{\sigma}$ such that the double partition $(\{A_i\}, \{A_{rs}\})$ of $[k]$ is trivial, i.e. $\{A_{rs} \}$ consists of the single set $[k]$.  Borrowing the result from Lemma \ref{s=1}, we see that $\displaystyle{\lim_{n \to \infty} \mathbb{E}[Y_{n,k}] = 0}$. Plugging $m=1$ into the expression \eqref{numform} then proves Theorem \ref{simplifythm} for $Y_{n,k}^{\textrm{tuple}}$.

\section{The $k$-subset representations} \label{ksubset}

Now we prove Theorem \ref{simplifythm} for $Y_{n,k}^{\textrm{set}}$. For each $\sigma \in \mathfrak{S}_n$, let $\sigma_k^{\textrm{set}}$ be the permutation of size $\displaystyle{ \binom{n}{k} }$ corresponding to $\rho_{n,k}^{\textrm{set}}(\sigma)$.

Similar to the ordered tuple case, the eigenvalue distribution is given by \begin{equation} \label{Xnkset} X_{n,k}^{\textrm{set}} = \binom{n}{k} (\beta - \alpha) + \sum_{j} C_{j,k}^{(n), \textrm{set}} (\{j \alpha\} - \{j \beta \}) 
\end{equation} where $C_{j,k}^{(n), \textrm{set}}(\sigma)$ is the number of cycles in $\sigma_k^{\textrm{set}}$ of length $j$.

As in the previous section, we will see that almost all of the contribution to the sum in $Y_{n,k}^{\textrm{set}}$ comes from the subsets $\{t_1,...,t_k\}$ such that $t_1,...,t_k$ are all in the same cycle of $\sigma$.  Although the argument is similar to the ordered tuple case, a few subtleties arise.
      
For each $\sigma \in \mathfrak{S}_n$, we wish to partition $Q_{n,k}^\textrm{set}$  according to the number of elements $t_i$ in each $\sigma$-cycle.  Unlike the ordered tuple case, instead of double partitioning the set $[k]$ we proceed by directly double partitioning the integer $k$.  

\begin{definition} Let $k = k_1+...+k_m$ such that $k_1 \ge ...\ge k_m \ge 1$ be a partition of the integer $k$.  Then, for each part $k_r$ choose a subpartition $\displaystyle{k_r = \sum_{s=1}^{p(r)} k_{rs}}$ such that $k_{r1} \ge ... \ge k_{r, p(r)} \ge 1$.  We can also denote the subpartition by a sequence $(c_{r,1},...,c_{r,k_r})$ such that $\displaystyle{\sum_{i=1}^{k_r} i c_{r,i} = k_r}$.  Here, $c_{r,i}$ represents the number of subparts of $k_r$ of size $i$.  We will call the array $(k_{rs})$ where $1 \le r \le m$ and $1 \le s \le p(r)$ a double partition of $k$.
\end{definition}

We can define the following subsets of $Q_{n,k}^\textrm{set}$ analogously to Definition \ref{deftuples}:

\begin{definition}
Let $(k_{rs})$ be a double partition of $k$ such that $r$ runs over the range $1 \le r \le  m$.  Choose a sequence $(i_1,...,i_m)$ of distinct integers from $[n]$.  Then let $T_{(k_{rs})}^{\sigma, (i_j) }$ denote the set of $k$-subsets $\{t_1,...,t_k\}$ such that for $1 \le r \le m$ and $1 \le s \le p(r)$,  $k_{rs}$ of the elements $t_i$ lie in the same $\sigma$-cycle of length $i_r$.  Taking the union over $\sigma$-cycle lengths, we also define \begin{equation} T_{(k_{rs})}^{\sigma} = \bigcup_{i_1 \neq... \neq i_m} T_{(k_{rs})}^{\sigma, (i_j) } 
\end{equation} 
\end{definition}

For each $\sigma \in \mathfrak{S}_n$, the set $\displaystyle{\Big\{ T_{(k_{rs})}^{\sigma}: (k_{rs}) \text{ a double partition of } k \Big\} }$ forms a partition of $Q_{n,k}^\textrm{set}$.  As before, we can consider each part $T_{(k_{rs})}^{\sigma} $ individually since the number of parts in this partition is only a function of $k$ and does not grow with $n$. 

To write the formula analogous to \eqref{numform} for the number of $\sigma_k^{\textrm{set}}$-cycles formed from $k$-subsets $\{t_1,...,t_k\} \in T_{(k_{rs})}^{\sigma, (i_j) }$, it will be helpful to introduce the notion of binary necklaces from combinatorics.  

\begin{definition} A binary necklace of length $n$ is an equivalence class of strings of 0's and 1's of length $n$ that are identified under rotation i.e. in the same orbit under action of the cyclic group $\mathbb{Z}/n\mathbb{Z}$.  The period of a necklace is the size of the corresponding equivalence class of strings, i.e. the period of a representative string.  

Let $N_{i,k}$ be the number of binary length $i$ necklaces with exactly $k$ ones and let $N_{i, k}^{d}$ denote the number of such necklaces of period $d$.  Finally, let $L_{i,k} := N_{i, k}^{i}$ denote the number of aperiodic necklaces of length $i$ with $k$ ones.  
\end{definition}
We will see shortly that for large $n$, almost all necklaces are aperiodic, i.e. have period $n$.  

For each subset $\{t_1,...,t_k\} \in T_{(k_{rs})}^{\sigma, (i_j) }$, for $1 \le r \le m$ and $1 \le s \le p(r)$, let $C_{rs}^{ \{t_i \}}$ be (one of) the $\sigma$-cycle(s) of length $i_r$ containing $k_{rs}$ elements of $\{t_1,...,t_k\}$.  We can identify $C_{rs}^{ \{t_i \}}$ with a binary necklace by giving the label 1 to the $k_{rs}$ elements of $\{t_1,...,t_k\}$ that lie in the cycle and giving the label 0 to the rest of the numbers in the cycle.  For example, if $C_{rs}^{ \{t_i \}}$ is the cycle $(3 5 2 9 1 4)$ and the subset of $k_{rs}$ elements of $\{t_1,...,t_k\}$ that lie in the cycle is $\{2,4\}$, then the induced binary necklace is $001001$.  Let $d_{rs}^{ \{t_i \}}$ denote the period of the binary necklace $C_{rs}^{ \{t_i \}}$. Then we see that each $k$-subset $\{t_1,...,t_k\} \in T_{(k_{rs})}^{\sigma,  (i_j) }$ lies in a $\sigma_k^{\textrm{set}}$-cycle of length $\big[\big(d_{rs}^{ \{t_i \}} \big) \big]$ where $\big[ (a_n) \big]$ denotes the least common multiple of all the elements in the sequence $(a_n)$.

Fix a double partition $(k_{rs})$ of $k$ and let $(i_1,...,i_m)$ be a sequence of $\sigma$-cycle lengths.  Let $(d_{rs})$ be an array of non-negative integers where the indices $r$ and $s$ run over the same range as $(k_{rs})$.  Then the $k$-subsets $\{t_1,...,t_k \} \in T_{(k_{rs})}^{\sigma, (i_j) }$ such that $d_{rs}^{ \{t_i\}} = d_{rs}$ form \begin{equation} \frac{1}{[(d_{ab})]} \prod_{r=1}^m \frac{\big(C_{i_r}^{(n)}(\sigma) \big)^{\underline{p(r)}}}{\prod_{i=1}^{k_r} c_{r,i}! } \prod_{s=1}^{p(r)} d_{rs} N_{i_r, k_{rs}}^{d_{rs}} \end{equation} 
$\sigma_k^{\textrm{set}}$-cycles of length $[(d_{ab})]$.

Note that $\displaystyle{N_{i_r, k_{rs}}^{d_{rs}} = L_{d_{rs}, \frac{k_{rs}d_{rs}}{i_r}}}$.  Clearly, this can only be non-zero if $i_r \bigm| k_{rs}d_{rs}$, i.e. $i_r f_{rs} = k_{rs}d_{rs}$ for some integer $f_{rs}$.  Since $d_{rs} \bigm| i_r$, we have $i_r = d_{rs} g_{rs}$ for some integer $g_{rs}$.  Putting this together, $k_{rs} = f_{rs} g_{rs}$.  Also, we have the trivial bound $\displaystyle{L_{i, k} \le \frac{1}{i} \binom{i}{k}}$.

Let $\mathbf{D}$ denote the set of all arrays of non-negative integers $(d_{rs})$ such that $N_{i_r, k_{rs}}^{d_{rs}} \neq 0$, i.e. such that each array entry $d_{rs}$ is a valid period.  Let $\mathbf{G}$ denote the set of all arrays of non-negative integers $(g_{rs})$ such that $g_{rs} \bigm| k_{rs}$.  (For both arrays, the indices run over the range $1 \le r \le m$ and $1 \le s \le p(r)$).  Then the number of $\sigma_k^{\textrm{set}}$ cycles formed from the elements of $T_{(k_{rs})}^{\sigma, (i_j) }$ is 
\begin{equation} \label{subsetform}
\sum_{(d_{ab}) \in \mathbf{D}} \frac{1}{[(d_{ab})]} \prod_{r=1}^m \frac{\big(C_{i_r}^{(n)}(\sigma) \big)^{\underline{p(r)}}}{\prod_{i=1}^{k_r} c_{r,i}! } \prod_{s=1}^{p(r)} d_{rs} N_{i_r, k_{rs}}^{d_{rs}}
\end{equation}

We have the following lemma analogous to Lemma \ref{tuplebound}.

\begin{lemma}\label{subsetbound} Let the array $(k_{rs})$ where $1 \le r \le m$ and $1 \le s \le p(r)$ be a double partition of $k$.  If $\displaystyle{\sum_{r=1}^m p(r) > 1}$, then \begin{equation} \lim_{n \to \infty} \frac{1}{n^{k-1}} \mathbb{E} \bigg[\sum_{i_1 \neq... \neq i_m} \sum_{(d_{ab}) \in \mathbf{D}} \frac{1}{[(d_{ab})]} \prod_{r=1}^m \frac{\big(C_{i_r}^{(n)}(\sigma) \big)^{\underline{p(r)}}}{\prod_{i=1}^{k_r} c_{r,i}! } \prod_{s=1}^{p(r)} d_{rs} N_{i_r, k_{rs}}^{d_{rs}} \bigg] = 0 \end{equation}   
\end{lemma}  
  
\begin{proof} 
\begin{equation*}
\begin{aligned}
&\sum_{(d_{ab}) \in \mathbf{D}} \frac{1}{[(d_{ab})]} \prod_{r=1}^m \frac{\big(C_{i_r}^{(n)}(\sigma) \big)^{\underline{p(r)}}}{\prod_{i=1}^{k_r} c_{r,i}! } \prod_{s=1}^{p(r)} d_{rs} N_{i_r, k_{rs}}^{d_{rs}}   \\
\le & \sum_{(d_{ab}) \in \mathbf{D}} \frac{1}{[(d_{ab})]} \prod_{r=1}^m \frac{\big(C_{i_r}^{(n)}(\sigma) \big)^{\underline{p(r)}}}{\prod_{i=1}^{k_r} c_{r,i}! } \prod_{s=1}^{p(r)} \binom{d_{rs} }{ \frac{k_{rs}d_{rs}}{i_r}}   \\
= & \sum_{(g_{ab}) \in \mathbf{G}} \frac{1}{[(i_a/g_{ab} ) ]} \prod_{r=1}^m \frac{\big(C_{i_r}^{(n)}(\sigma) \big)^{\underline{p(r)}}}{\prod_{i=1}^{k_r} c_{r,i}! } \prod_{s=1}^{p(r)} \binom{i_r/g_{rs} }{ k_{rs}/g_{rs} }  \\
< & \frac{1}{[i_1,...,i_m]} \sum_{(g_{ab}) \in \mathbf{G}} \prod_{r=1}^m \big(C_{i_r}^{(n)}(\sigma) \big)^{\underline{p(r)}} \prod_{s=1}^{p(r)} \binom{i_r}{ k_{rs}} k_{rs} \\
= & \frac{|\mathbf{G}|}{[i_1,...,i_m]} \prod_{r=1}^m \big(C_{i_r}^{(n)}(\sigma) \big)^{\underline{p(r)}} \prod_{s=1}^{p(r)} \binom{i_r }{ k_{rs}} k_{rs}
\end{aligned}
\end{equation*}
since $\displaystyle{\binom{n}{ m} \le \binom{a n}{ a m}}$ for $a \ge 1$ and $[p m_1, m_2] \le p [m_1, m_2]$.

The result then follows from  Lemma \ref{tuplebound}. 

\end{proof}

Thus, just as for the $k$-tuple case, the only $\sigma_k^{\textrm{set}}$-cycles that contribute to $Y_{n,k}^{\textrm{set}}$ in the limit $n \to \infty$ are those formed from $k$-subsets in $T_{(k_{rs})}^{\sigma}$ such that the double partition $(k_{rs})$ of $k$ is trivial, i.e. $(k_{rs})$ consists of just one part of size $k$.  Since $\displaystyle{N_{j, k}^{d} \le \frac{1}{d} \binom{d}{ k d/j}}$, it is easy to see that for large $j$, almost all necklaces of length $j$ with $k$ ones are aperiodic, i.e. both $N_{j, k}$ and $L_{j, k}$ are asymptotically $\displaystyle{\frac{j^{k-1}}{k!} + O(j^{k-2})}$.  Plugging $m=1$ into the expression \eqref{subsetform} then proves Theorem \ref{simplifythm} for $Y_{n,k}^{\textrm{set}}$.   

\begin{remark}
Exact formulas for $N_{n,i}$ and $L_{n,i}$ are known:

\begin{equation} L_{n,i} = \frac{1}{n} \sum_{d \mid (n, i)} \mu(d) \binom{n/d}{ i/d} \end{equation}
\begin{equation} N_{n,i} = \frac{1}{n} \sum_{d \mid (n,i)} \varphi(d) \binom{n/d }{ i/d} \end{equation}
where $\mu(d)$ is the M\"{o}bius function and $\varphi$ is Euler's totient function. Derivation of these formulas and other results about necklaces can be found in e.g. \cite{bender, ruskey, sawada}.

\end{remark}

\section{The $S^{(n-k,1^k)}$ irreducible representation} \label{irrep}

In this section, we finally prove Theorem \ref{simplifythm} for the $S^{(n-k,1^k)}$ irreducible representation of the symmetric group $\mathfrak{S}_n$.  First, we briefly review some basic facts from the representation theory of symmetric groups.

\subsection{Basics of symmetric group theory}

It is well known that every complex representation of a finite group is completely reducible, i.e. is the direct sum of irreducible representations.  This follows from the fact that finite-dimensional unitary representations of any group are completely reducible and Weyl's unitary trick which shows that every finite dimensional representation of a finite group is unitarisable.  Then the eigenvalue distribution of any finite group representation is simply a mixture of the eigenvalue distributions for each irreducible representation in the direct sum.  
 
Thus, to understand the eigenvalue distributions of representations of the symmetric group, another perspective is to try to understand the irreducible representations.  These representations are indexed by the partitions of $n$, often denoted $\lambda \vdash n$.  We can visualize a partition $\lambda$ by drawing its diagram, which is a configuration of boxes arranged in left-justified rows such that there are $\lambda_i$ boxes in the $i^{\textrm{th}}$ row.

\begin{definition} Given a partition $\lambda \vdash n$, a Young tableau of shape $\lambda$ is obtained by placing the integers $[n]$ into the diagram for $\lambda$ (so that each number appears exactly once).  Clearly, there are $n!$ Young $\lambda$-tableaux.  A standard Young tableau is a tableau such that the entries are strictly increasing in each row and each column.  If $\lambda, \mu \vdash n$, a semistandard tableau of shape $\lambda$ and type $\mu$ is a tableau where the entries are weakly increasing along each row and strictly increasing down each column such that the number $i$ appears $\mu_i$ times. 
\end{definition}
One can consider an equivalence relation on the set of $\lambda$-tableaux such that $t_1 \sim t_2$ if $t_1$ and $t_2$ contain the same elements in each row.  Each equivalence class $\{t \}$ under this relation is called a tabloid.  Thus, a tabloid is a tableau that only cares about rows.  

The action of $\mathfrak{S}_n$ on tabloids induces the permutation representation  on a vector space with basis $e_{\{t \}}$ in the usual way.  These $\frac{n!}{\lambda_1!...\lambda_r!}$ dimensional representations are denoted by $M^\lambda$ for each partition $\lambda \vdash n$ and called the permutation module corresponding to $\lambda$.  
Using this terminology, the permutation representation on ordered $k$-tuples is equivalent to the permutation module $M^{(n-k,1^k)}$ and the permutation representation on unordered $k$-subsets is equivalent to the permutation module $M^{(n-k,k)}$.   

One can find the irreps in the permutation modules $M^\lambda$.  Define for each tableau $t$ a polytabloid $e_t \in M^\lambda$ by $\displaystyle{e_t = \sum_{\pi \in C_t} \sgn(\pi) e_{\pi\{t\}}}$ where $C_t$ is the subgroup of $\mathfrak{S}_n$ that stabilizes columns of $t$.  Then the subspace of $M^\lambda$ spanned by the $\{e_t\}$ is called the Specht module $S^\lambda$.  As $\lambda$ ranges over the partitions of $n$, $S^\lambda$ give all the irreps of $\mathfrak{S}_n$.  The set of polytabloids $\{e_t : t \text{ is a standard } \lambda\text{-tableau}\}$ is a basis for $S^\lambda$.  Thus, we see that the dimension of $S^\lambda$ is the number of standard $\lambda$-tableaux.  The celebrated hook-length formula gives a formula for this number. 

Young's rule gives a method of determining which irreducible subrepresentations are present in the permutation module $M^\lambda$. 

\begin{lemma}[Young's Rule] \label{Young}
The multiplicity of $S^\lambda$ in $M^\mu$ is the Kostka number $K_{\lambda \mu}$, which is the number of semistandard tableau with shape $\lambda$ and type $\mu$. 
\end{lemma}
For a proof of Lemma \ref{Young}, see e.g. \cite[Prop. 7.18.7]{stanley}.  By Young's rule, we see for instance that $\displaystyle{M^{(n-2,1,1)} = S^{(n)} \oplus 2S^{(n-1,1)} \oplus S^{(n-2,2)} \oplus S^{(n-2,1,1)}}$.  In general, $S^{(n-k,1^k)}$ appears as an irrep of $M^{(n-k,1^k)}$ with multiplicity 1.  

The decomposition of $M^{(n-k,k)}$ into irreducibles is particularly easy to describe.  We have $M^{(n-k,k)} = S^{(n)} \oplus S^{(n-1,1)} \oplus...\oplus S^{(n-k,k)}$.  

More information about symmetric group theory can be found in any number of references.  A few are \cite{diaconisbook, james, sagan, silberstein}.  

\subsection{Eigenvalue distributions of irreducible representations}

Stembridge \cite{stembridge} has found an explicit formula for the eigenvalues of any irreducible representation of the symmetric group in terms of Young tableaux.  In the following, we borrow terminology from \cite{stembridge}.  First, we introduce the notion of a descent set. 

\begin{definition}
Let $T$ be a standard Young tableau.  If $k+1$ appears in a row strictly below $k$ in $T$, then $k$ is said to be a descent of $T$.  We write $D(T)$ for the set of descents in $T$.  
\end{definition}
Figures \ref{fig:sub1} and \ref{fig:sub2} give the descent sets for a standard tableau of shape $(6,1,1,1)$ and another of shape $(4,3,2)$.  

Let $\mu \vdash n$ be the cycle type (i.e. list of cycle lengths in the cycle decomposition in non-increasing order) of $\sigma \in \mathfrak{S}_n$.  Let $\rho_{n,1}^{\textrm{tuple}}$ be the defining representation of $\mathfrak{S}_n$.  Then we define $b_\mu = (b_\mu(1),...,b_\mu(n))$ to be the vector of eigenangles of $\rho_{n,1}^{\textrm{tuple}}(\sigma)$ listed by cycle.  For example, $\displaystyle{b_{(4,4,3,2)} = \left(\frac{1}{4},\frac{2}{4},\frac{3}{4},1,\frac{1}{4},\frac{2}{4},\frac{3}{4},1,\frac{1}{3},\frac{2}{3},1,\frac{1}{2},1 \right)}$.  

Now we can state Stembridge's formula for the eigenvalues:

\begin{theorem}[Stembridge] \label{stembridge}
Let $\rho^\lambda$ be the representing map corresponding to the irrep $S^\lambda$.  The eigenangles of $\rho^\lambda(\sigma)$ (counted with multiplicity) are indexed by standard Young $\lambda$-tableaux $T$ and given by $\displaystyle{\sum_{i \in D(T)} b_\mu(i)}$ where the sum is taken mod 1.
\end{theorem}

\begin{figure}
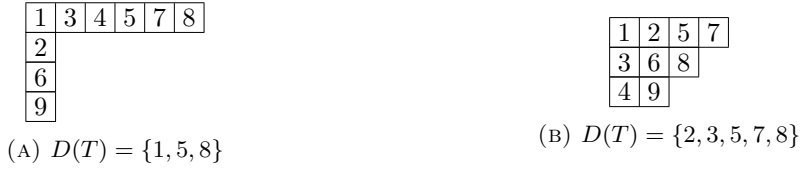

\centering
\begin{subfigure}{.5\textwidth}
\centering
\young(134578,2,6,9)
\caption{$D(T) = \{1, 5, 8\}$}
\label{fig:sub1}
\end{subfigure}%
\begin{subfigure}{.5\textwidth}
\centering
\young(1257,368,49)
\caption{$D(T) = \{2, 3, 5, 7, 8\}$}
\label{fig:sub2}
\end{subfigure}
\caption{Descent sets of two Young tableau}
\end{figure}
   
Now, we turn to the eigenvalue distribution of $\rho^{(n-k,1^k)}$.  It is easy to check that

\begin{proposition} \label{stem}
For $T$ running over all standard Young tableaux of shape $(n-k,1^k)$, we have \[ \{D(T)\} = Q_{n-1,k}^{\textrm{set}} \]
\end{proposition}
   
Let $\mu = (\mu_1,\mu_2,...,\mu_l)$ be the cycle type of $\sigma \in \mathfrak{S}_n$.  For any set of sets (or tuples) $U$, define 
\begin{equation} \label{etuple} E_n^U(\sigma) := \bigg\{\sum_{i \in S} b_\mu(i): S \in U \bigg\} \end{equation} 

By Theorem \ref{stembridge} and Proposition \ref{stem}, the multiset of eigenangles of $\rho^{(n-k,1^k)}(\sigma)$ is $\displaystyle{E_n^{Q_{n-1,k}^{\textrm{set}} }(\sigma) }$.  It will be easier to work our way towards $\displaystyle{E_n^{Q_{n-1,k}^{\textrm{set}} } }$ by first considering $\displaystyle{E_n^{Q_{n,k}^{\textrm{tuple*}} }}$ where $Q_{n,k}^\textrm{tuple*}$ is the set of $k$-tuples allowing repeats.

By the same reasoning as for $Y_{n,k}^{\textrm{tuple}}$, it is easy to see (just replace $\prod i_r^{\underline{k_{rs}}}$ with $\prod i_r^{k_{rs}}$ in Lemma \ref{tuplebound}):

\begin{lemma} \label{Qtuplestar}
As $n \to \infty$, $\displaystyle{ \frac{\big| E_n^{Q_{n,k}^\textrm{tuple*}} \cap I \big| - \mathbb{E}\big| E_n^{Q_{n,k}^\textrm{tuple*}} \cap I \big|}{n^{k-1}}}$ has the same limiting law as $Y_{n,k}$.
\end{lemma}

We now want to show that the same is true for $E_n^{Q_{n-1,k}^\textrm{tuple}}$.  Note that we have the decomposition \begin{equation}\label{decomposition} Q_{n,k}^\textrm{tuple*} = Q_{n,k}^\textrm{tuple} \bigcup Q_{n,k}^\textrm{duplicate} \bigcup Q_{n,k}^\textrm{rest} 
\end{equation} where $Q_{n,k}^\textrm{duplicate}$ contains the tuples with exactly two identical entries and $Q_{n,k}^\textrm{rest}$ contains the rest of the tuples in $Q_{n,k}^\textrm{tuple*}$.  

We have the recursive relation:
\begin{equation}\label{recursive} E_n^{Q_{n,k}^\textrm{tuple}} = E_n^{Q_{n-1,k}^\textrm{tuple}} \bigcup \big(E_n^{Q_{n-1,k-1}^\textrm{tuple}}\big)^{\bullet k} \end{equation} where $\big( E_n^{Q_{n-1,k-1}^\textrm{tuple}} \big)^{\bullet k}$ denotes the multiset containing $k$ copies of $E_n^{Q_{n-1,k-1}^\textrm{tuple}}$.

We also have the following decomposition of $E_n^{Q_{n,k}^\textrm{duplicate}}$: 

\begin{lemma} \label{duplicate}
The multiset $E_n^{Q_{n,k}^\textrm{duplicate}}$ is the union of $n {k \choose 2}$ rotated copies of $E_n^{Q_{n-1,k-2}^\textrm{tuple}}$.
\end{lemma}

\begin{proof}
Note that if $S \in Q_{n,k}^\textrm{tuple}$ is a $k$-tuple containing any element $j$ such that $b_\mu(j) = 1$ and $S^{-j}$ is the $(k-1)$-tuple gotten from $S$ by omitting $j$, then \begin{equation} \label{triv} \sum_{i \in S} b_\mu(i)=\sum_{i \in S^{-j}} b_\mu(i)
\end{equation}  
Let  $Q_{n,k}^{\textrm{tuple}-j}$ denote the set of $k$-tuples not containing the element $j$.  Then using \eqref{triv}, we see that for $1 \le j \le n$, the multisets $E_n^{Q_{n,k}^{\textrm{tuple}-j}}$ are all rotations of $E_n^{Q_{n-1,k}^{\textrm{tuple}}}$.  Since there are $n {k \choose 2}$ ways to pick two entries of a $k$-tuple in $Q_{n,k}^\textrm{duplicate}$ and assigning the same value $j$, the result then follows.
\end{proof}
Together, \eqref{recursive} and Lemma \ref{duplicate} show that $E_n^{Q_{n,k}^\textrm{tuple*}}$ is the union of $E_n^{Q_{n-1,k}^\textrm{tuple}}$, $k$ copies of $E_n^{Q_{n-1,k-1}^\textrm{tuple}}$, $n {k \choose 2}$ rotated copies of $E_n^{Q_{n-1,k-2}^\textrm{tuple}}$ and a set with cardinality of order $O(n^{k-2})$ that we can ignore.  Then by Lemma \ref{Qtuplestar} and inducting on $k$, we have 
\begin{lemma} \label{Qtuple}
As $n \to \infty$, $\displaystyle{ \frac{\big| E_n^{Q_{n-1,k}^\textrm{tuple}} \cap I \big| - \mathbb{E}\big| E_n^{Q_{n-1,k}^\textrm{tuple}} \cap I \big|}{n^{k-1}}}$ has the same limiting law as $Y_{n,k}$.
\end{lemma}
Since $E_n^{Q_{n-1,k}^\textrm{tuple}}$ just consists of $k!$ copies of  $E_n^{Q_{n-1,k}^\textrm{set}}$, this proves Theorem \ref{simplifythm} for $Y_{n,k}^\textrm{irrep}$.

\section{Equidistributed sequences} \label{equidistributed sequences}

In this section, we review some of the theory of uniform distribution mod 1 that will be important in the sequel.  This material is all contained in Kuipers and Neiderreiter's book \cite{kuipers}, which contains many other interesting results on equidistribution.   

It will be convenient to identify the interval $[0,1]$ with the 1-dimensional torus (circle) $\mathbb{T}^1$ by identifying the two endpoints.  Since $\mathbb{T}^1$ is a group under addition, this will obviate the need to take fractional parts. 
\begin{definition}\label{equidistributed}
A sequence $(x_n)_{n \in \mathbb{N}}$ of elements of $\mathbb{T}^d$ is said to be equidistributed or uniformly distributed if for every box $\displaystyle{B = \prod_{i=1}^d [a_i, b_i] }$ such that $0 \le a_i < b_i \le 1$, we have  \begin{equation}\lim_{n \to \infty} \frac{A(B; n)}{n} = \prod_{i=1}^d (b_i - a_i) \end{equation} where $A(B; n)$ counts the number of elements of the sequence $(x_1,...,x_n)$ in the box $B$. 
\end{definition}

We have the following important criterion for equidistribution first formulated by Hermann Weyl.

\begin{theorem}[Weyl's Criterion]\label{weylcriterion}
The sequence $(x_j)_{j \in \mathbb{N}}$ of elements in $\mathbb{T}^d$ is equidistributed if and only if for each nonzero element $h \in \mathbb{Z}^d$, $\displaystyle{\lim_{n \to \infty} \frac{1}{n} \sum_{j=1}^n e^{2 \pi i h \cdot x_j} = 0}$.  

\end{theorem}

From this, it is easy to establish \cite[Thm. 1.6.4]{kuipers}:
\begin{theorem}[Weyl's Equidistribution Theorem]\label{weylthm}
If $\gamma$ is irrational, then the sequence $(n \gamma)_{n \in N}$ is equidistributed.  More generally, if $p$ is a polynomial with at least one nonconstant irrational coefficient, then the sequence $(p(n))_{n \in N}$ is equidistributed.     
\end{theorem}

Weyl's criterion only gives a qualitative asymptotic condition for equidistribution.  It will also be useful to have quantitative bounds on the rate of convergence to equidistribution. 

\begin{definition}\label{multidiscrepancy}
Let $J$ be the set of $d$-dimensional boxes of the form $\displaystyle{\prod_{i=1}^d [a_i, b_i] }$ where $0 \le a_i < b_i \le 1$ and let $\mathcal{P}_n$ be a multiset or sequence of $n$ elements of the $d$-dimensional torus $\mathbb{T}^d$. 
The multidimensional discrepancy is given by \[D(\mathcal{P}_n) := \sup_{B \in J} \bigg|\frac{A(B; n)}{n} - \prod_{i=1}^d (b_i - a_i) \bigg| \] where $A(B; n)$ counts the number of elements of $\mathcal{P}_n$ in the box $B$.  
\end{definition} 

\begin{remark}
By \cite[Thm. 2.1.1]{kuipers}, a sequence $(x_n)_{n \in \mathbb{N}}$ is equidistributed if and only if \sloppy $\displaystyle{ \lim_{n \to \infty} D(x_1,...,x_n) = 0}$.

\end{remark}

\begin{definition}

If $\omega$ is an infinite sequence, let $D_{i,n}(\omega)$ be the discrepancy of the $(i+1)^\textrm{st}$ through $(i+n)^\textrm{th}$ terms of the sequence.  If $D_{i,n}(\omega) \to 0$ uniformly in $i$ as $n \to \infty$, we say that the sequence $\omega$ is \textit{well-distributed}. 
\end{definition}

\begin{remark}
Note that for the sequence $x_n = n \alpha +\beta$, the discrepancy $D_{i,n}(x)$ only depends on $n$ since the subsequence $x_{i+1},...,x_{i+n}$ is just a translate of the sequence $x_1,...,x_n$.  Thus, $(x_n)$ is well distributed.  In fact, using the van der Corput lemma, one can show that if $p$ is a polynomial with at least one nonconstant irrational coefficient, then the sequence $(p(n))_{n \in N}$ is well-distributed.
\end{remark}

The following definition is useful to state various estimates for the discrepancy:
\begin{definition}\label{defirrational}
The irrationality measure, $\mu(r)$, of a real number $r$ is given by   \[\mu(r)= \inf \left\{ \lambda\colon \left\lvert r-\frac{p}{q}\right\rvert < \frac{1}{q^{\lambda}} \text{ has only finitely many integer solutions in p and q }\right\} \]
\end{definition}
 
\begin{theorem}[{\cite[Thm. 2.3.2]{kuipers}}] \label{1Ddisc}
Let $\alpha$ have irrationality measure $\lambda$ and let $x_n = n \alpha + \beta$.  Then for every $\varepsilon > 0$, \begin{equation} D_{i,n}(x) = O(n^{-\frac{1}{\lambda-1} + \varepsilon}) \end{equation}     
\end{theorem}
 
\section{Moment method for $k=1$ case} \label{k=1}
Before applying the moment method to find the limiting distribution of $Y_{n,k}$ for $k > 1$, let us first consider the $k=1$ case of permutation matrices.  Then the appropriate scaled count of eigenangles in the interval $I=(\alpha, \beta)$ is 
\begin{equation} Y_{n,1}^\textrm{tuple} := \frac{X_{n,1}^\textrm{tuple} - \mathbb{E}[X_{n,1}^\textrm{tuple}]}{\sqrt{\log n}} = \frac{1}{\sqrt{\log n}} \sum_{j=1}^n \bigg(C_j^{(n)} - \frac{1}{j} \bigg) (\{j\alpha \} - \{j\beta \})
\end{equation}

Wieand \cite{wieand} (for $\theta=1$) and Ben Arous and Dang \cite{arousdang} (for general $\theta > 0$) use the Feller coupling \cite{feller} along with the CLT to show limiting normality of $Y_{n,1}^\textrm{tuple}$.  The Feller coupling is a way of constructing random permutations using sums of Bernoulli random variables that allows for quantitative bounds on the distance between $C_j^{(n)}$ and independent Poisson variables $W_j$ with parameter $1/j$.  Using this coupling, it turns out that the asymptotic behavior of $Y_{n,1}^\textrm{tuple}$ is unchanged if one replaces the dependent variables $C_j^{(n)}$ with the independent variables $W_j$.  See \cite{wieand} for details.  In this section, we will apply the method of moments to rederive this result.    

To make computing the moments simpler, we assume  $\alpha$ and $\beta$ are irrationals linearly independent over $\mathbb{Q}$ of finite irrationality measure.  By Khintchine's theorem, the set of numbers with irrationality measure greater than 2 has Lebesgue measure 0, so this is not a very restrictive condition.  It follows from Theorem \ref{1Ddisc} and \cite[Thm. 3]{wieand} that \begin{equation} \label{finiteness} \bigg|\sum_{j=1}^n \frac{1}{j} (\{j\alpha \} - \{j\beta \})\bigg| < C 
\end{equation} for some absolute constant $C$.  With this additional finiteness restriction on $\alpha$ and $\beta$, it suffices then to show that  \begin{equation} Z_n := \frac{1}{\sqrt{\log n}} \sum_{j=1}^n C_j^{(n)} (\{j\alpha \} - \{j\beta \}) \end{equation} limits to a normal distribution.  We wish to establish the following proposition, which states the convergence of the moments of $Z_n$ to those of a centered normal distribution with variance $\theta/6$.

\begin{proposition} \label{momentsk=1}
The odd moments of $Z_n$ limit to 0.  For even $m$,
\begin{equation}\lim_{n \to \infty} \mathbb{E}[(Z_n)^m] = \frac{m!}{2^{m/2}} \frac{1}{(m/2)!} \left(\frac{\theta}{6}\right)^{m/2}
\end{equation} 
\end{proposition}

For a partition $m = m_1+...+m_s$ such that $m_1 \ge... \ge m_s$, let $c_i$ be the number of parts $m_j$ equal to $i$, so that $\displaystyle{\sum_{i=1}^m i c_i = m}$.  By the multinomial theorem, 

\begin{multline} \label{multinomial}
(Z_n)^m =\frac{1}{(\log n)^{m/2}} \sum_{(m_i) \vdash m} \binom{m}{ m_1,...,m_s} \frac{1}{\prod_{l = 1}^m c_l!} \\ \sum_{j_1 \neq ... \neq j_s}  \prod_{l=1}^s \Big(C_{j_l}^{(n)}\Big)^{m_l} (\{j_l\alpha \} - \{j_l\beta \})^{m_l} 
\end{multline}

\begin{remark}
The coefficients $\displaystyle{ \binom{m}{ m_1,...,m_s} \frac{1}{\prod_{l = 1}^m c_l!} }$ are the so-called Fa\`a di Bruno coefficients which arise in the Fa\`a di Bruno formula \cite{faa} for derivatives as well as the expansion of Bell polynomials.   
\end{remark}
First, we collect a few estimates that we will need:

\begin{lemma} \label{estimates}
\begin{enumerate}
\item[]
\item For large $n$ and $s \ge 1$, \begin{equation}
\sum_{\substack{j_1+...+j_s = n \\ j_l \ge 1}} \frac{1}{j_1...j_s} = O\left(\frac{1}{n} (\log n)^{s-1} \right)
\end{equation}

\item For $\theta > 0$ and integer $1 \le i \le n$, \begin{equation}\label{thetabound} \prod_{j=0}^{i-1} \frac{n-j}{\theta+n-j-1} \le A(\theta) \left( \frac{\theta+n-1}{\theta+n-i} \right)^{1-\theta}
\end{equation} for some constant $A(\theta)$.  Moreover, if $i$ is such that  $n-i = \Omega(n)$, then \begin{equation}\label{thetasim} \prod_{j=0}^{i-1} \frac{n-j}{\theta+n-j-1} \sim \left( \frac{\theta+n-1}{\theta+n-i} \right)^{1-\theta}
\end{equation} as $n \to \infty$.
\end{enumerate}
\end{lemma}

\begin{proof}
Inequality (1) follows from a simple induction argument.  (2) follows from the fact that $\log(1+x) \le x$ for $x > 0$ and $\log(1+x) \approx x$ for small $x$.

\end{proof}

Using Lemma \ref{estimates}, we can compute the following limit:
\begin{lemma} \label{deltabound}
Let $\theta > 0$ and let $\delta > 0$ be small.  Then for $s \ge 1$
\begin{equation}
\frac{1}{(\log n)^{s}} \sum_{\substack{j_1+...+j_s \le n \\ j_1 \ge \lfloor \delta n \rfloor} } \prod_{i=0}^{\sum j_l -1} \frac{n-i}{\theta+n-i-1} \prod_{l=1}^s \left( \frac{\theta}{j_l} \right) = 0
\end{equation}
where all indices $1 \le j_1,...,j_s \le n$.
\end{lemma}
\begin{proof}

Recall Vinogradov's Big-Oh notation $\ll$ to denote inequality up to an absolute constant $C$ as $n \to \infty$.  Then

\begin{align}
&\nonumber\sum_{\substack{j_1+...+j_s \le n \\ j_1 \ge \lfloor \delta n \rfloor} } \prod_{i=0}^{\sum j_l -1} \frac{n-i}{\theta+n-i-1} \prod_{l=1}^s \left( \frac{\theta}{j_l} \right) \\
\ll &\nonumber\sum_{u=\lfloor \delta n \rfloor}^n \sum_{j_1=\lfloor \delta n \rfloor}^{u-1} \frac{1}{j_1}  \left( \frac{\theta+n-1}{\theta+n-u} \right)^{1-\theta} \sum_{j_2+...+j_s = u-j_1} \frac{1}{j_2...j_s} \\
\ll &\nonumber\sum_{u=\lfloor \delta n \rfloor}^n \sum_{j_1=\lfloor \delta n \rfloor}^{u-1} \frac{1}{j_1}  \left( \frac{\theta+n-1}{\theta+n-u} \right)^{1-\theta} \frac{1}{u-j_1} \log(u-j_1)^{s-2} \\
\ll &\nonumber\sum_{u=\lfloor \delta n \rfloor}^n \frac{1}{n^\theta} \frac{1}{(n-u)^{1-\theta}} (\log u)^{s-1} \\
\ll & (\log n)^{s-1}
\end{align}

\end{proof}

We are now ready to prove the following estimate involving the expectation $\displaystyle{\mathbb{E}\Big[\prod_{l=1}^s \Big(C_{j_l}^{(n)}\Big)^{\underline{r_l}}  \Big]}$. 
\begin{lemma} \label{Znlimitlemma}
Let $m = m_1+...+m_s$ be a partition of $m$ and $r_1,...,r_s$ be integers such that $1 \le r_l \le m_l$.  Then 
\begin{multline} \label{Znlimit}
\lim_{n \to \infty} \frac{1}{(\log n)^{m/2}} \sum_{j_1 \neq ... \neq j_s} \mathbb{E}\Big[\prod_{l=1}^s \Big(C_{j_l}^{(n)}\Big)^{\underline{r_l}} \Big] (\{j_l\alpha \} - \{j_l\beta \})^{m_l} \\ = \lim_{n \to \infty} \frac{1}{(\log n)^{m/2}} \sum_{j_1,...,j_s} \prod_{l=1}^s \left( \frac{\theta}{j_l} \right)^{r_l} (\{j_l\alpha \} - \{j_l\beta \})^{m_l}
\end{multline}
Moreover, the limit is 0 unless $r_l = 1$ and $m_l=2$ for all $l$.
\end{lemma}

\begin{proof}
By \eqref{diaconispoisson}, we have
\begin{multline} \label{cyclesum}
\sum_{j_1 \neq ... \neq j_s} \mathbb{E}\Big[\prod_{l=1}^s \Big(C_{j_l}^{(n)}\Big)^{\underline{r_l}} \Big] (\{j_l\alpha \} - \{j_l\beta \})^{m_l} \\ 
= \sum_{\substack{j_1 \neq ... \neq j_s \\ \sum_{l=1}^s j_l r_l \le n}}  \prod_{i=0}^{\sum j_l r_l -1} \frac{n-i}{\theta+n-i-1} \prod_{l=1}^s \left( \frac{\theta}{j_l} \right)^{r_l} (\{j_l\alpha \} - \{j_l\beta \})^{m_l} 
\end{multline}
If there exist parts in the partition of size 1, let $s'$ be such that $m_l = 1$ for $l > s'$.  Then by \eqref{finiteness}, we can bound \eqref{cyclesum} (up to a constant) by

\begin{gather} 
\nonumber \sum_{\substack{j_1 \neq ... \neq j_{s'} \\ \sum_{l=1}^{s'} j_l r_l \le n}}  \prod_{i=0}^{\sum j_l r_l -1} \frac{n-i}{\theta+n-i-1} \prod_{l=1}^{s'} \left( \frac{\theta}{j_l} \right)^{r_l} (\{j_l\alpha \} - \{j_l\beta \})^{m_l} \\
\ll \sum_{\substack{1 \le j_l \le \delta n \\ 1 \le l \le s'}} \frac{1}{j_1...j_{s'}} \ll (\log n)^{(m-1)/2}
\end{gather}
where we have used Lemma \ref{deltabound}.  Thus, both sides of \eqref{Znlimit} are clearly 0 when there is a part of size 1, and we can assume all parts are size at least 2.  Then $s \le m/2$.  By Lemma \ref{deltabound}, 
\begin{align}
&\nonumber \lim_{n \to \infty} \frac{1}{(\log n)^{m/2}}  \sum_{\substack{j_1 \neq ... \neq j_s \\ \sum_{l=1}^s j_l r_l \le n}}  \prod_{i=0}^{\sum j_l r_l -1} \frac{n-i}{\theta+n-i-1} \prod_{l=1}^s \left( \frac{\theta}{j_l} \right)^{r_l} (\{j_l\alpha \} - \{j_l\beta \})^{m_l} \\
=&\nonumber \lim_{n \to \infty} \frac{1}{(\log n)^{m/2}}  \sum_{\substack{1 \le j_l \le \delta n \\ 1 \le l \le s}} \prod_{i=0}^{\sum j_l -1} \frac{n-i}{\theta+n-i-1} \prod_{l=1}^s \left( \frac{\theta}{j_l} \right)^{r_l} (\{j_l\alpha \} - \{j_l\beta \})^{m_l}
\end{align}
This limit is clearly 0 unless $r_l = 1$ and $m_l=2$ for all $l$ in which case \eqref{Znlimit} follows by taking $\delta \to 0$.
\end{proof}

We can finally obtain a simplified expression for the sum in \eqref{multinomial}.
\begin{lemma} \label{Znstirlinglemma}
Let $m = m_1+...+m_s$ be a partition of $m$.  Then \begin{multline}\label{Znstirling}
\lim_{n \to \infty} \frac{1}{(\log n)^{m/2}} \sum_{j_1 \neq ... \neq j_s} \mathbb{E}\Big[\prod_{l=1}^s \Big(C_{j_l}^{(n)}\Big)^{m_l} \Big] (\{j_l\alpha \} - \{j_l\beta \})^{m_l} \\ = \lim_{n \to \infty} \frac{1}{(\log n)^{m/2}} \sum_{j_1,...,j_s} \prod_{l=1}^s \frac{\theta}{j_l} (\{j_l\alpha \} - \{j_l\beta \})^{m_l}
\end{multline}
where the limit is 0 unless $m_l=2$ for all $l$.
\end{lemma}

\begin{proof}
We rewrite $\displaystyle{\prod_{l=1}^s \Big(C_{j_l}^{(n)}\Big)^{m_l}}$ by making use of the following identity relating ordinary powers to falling factorial powers:
\begin{equation}\label{stirling} x^n = \sum_{r=0}^n \stirling{n}{r} x^{\underline{r}} \end{equation}
where the curly braces denote Stirling numbers of the second kind, i.e. the number of ways to partition $[n]$ into $r$ non-empty subsets.   

Then we have \begin{equation}\prod_{l=1}^s \Big(C_{j_l}^{(n)}\Big)^{m_l} = \sum_{1 \le r_l \le m_l} A_{(r_1,...,r_s)} \prod_{l=1}^s \Big( C_{j_l}^{(n)} \Big)^{\underline{r_l}} 
\end{equation}
for some constants $ A_{(r_1,...,r_s)}$ where note that $A_{(1,...,1)} = 1$.  The result follows from Lemma \ref{Znlimitlemma}.
\end{proof}

By logarithmic summability \cite[p. 1569]{wieand}, \begin{equation}\label{logarithmic} \lim_{n \to \infty} \frac{1}{\log n} \sum_{j=1}^n \frac{1}{j} (\{j\alpha \} - \{j\beta \})^{m} = \lim_{n \to \infty} \frac{1}{n} \sum_{j=1}^n (\{j\alpha \} - \{j\beta \})^{m}\end{equation}  We will need this limit for general $m$ in section \ref{k>1}.
\begin{lemma}\label{calculation} If $m$ is even,
\[\lim_{n \to \infty} \frac{1}{n} \sum_{j=1}^n (\{j\alpha \} - \{j\beta \})^{m} = \frac{2}{(m+1)(m+2)} \]  If $m$ is odd, the limit is 0.
\end{lemma}
\begin{proof}
The sequence $(j\alpha , j\beta )$ is uniformly distributed mod 1.  (See Definition \ref{equidistributed} and Theorem \ref{weylcriterion}).  Thus, by exercise 1.6.3 in \cite[p. 52]{kuipers}, for Riemann integrable functions $f(x,y)$, \[\lim_{n \to \infty} \frac{1}{n} \sum_{j=1}^n f(\{j\alpha \}, \{j\beta \}) = \int_0^1 \int_0^1 f(x,y)dx dy.\]   Therefore, \[\lim_{n \to \infty} \sum_{j=1}^n \frac{(\{j\alpha \} - \{j\beta \})^m}{n} = \mathbb{E} [(U_1 - U_2)^m] = \mathbb{E} \bigg[\sum_{i=0}^m \binom{m}{ i} U_1^i (-U_2)^{m-i}\bigg] \] where $U_1$ and $U_2$ are independent variables uniform on $[0,1]$.  If $m$ is odd, this expectation is 0 by symmetry.  If $m$ is even, the expectation equals 
\begin{equation*}
\begin{aligned}
\sum_{i=0}^m \binom{m}{ i} \frac{1}{(i+1)(m-i+1)} (-1)^{m-i} &= \sum_{i=0}^m \frac{1}{(m+1)(m+2)} \binom{m+2}{ i+1} (-1)^i \\ &= \frac{2}{(m+1)(m+2)} 
\end{aligned}
\end{equation*}
\end{proof}

Now combining \eqref{multinomial}, Lemma \ref{Znstirlinglemma}, and Lemma \ref{calculation}, we see that the odd moments of $Z_n$ converge to 0 and that even moments converge to \[\lim_{n \to \infty} \frac{m!}{2^{m/2}} \frac{1}{(m/2)!} \frac{1}{(\log n)^{m/2}} \bigg(\sum_{j=1}^n \frac{\theta}{j} (\{j\alpha \} - \{j\beta \})^2 \bigg)^{m/2} = \frac{m!}{2^{m/2}} \frac{1}{(m/2)!} \left(\frac{\theta}{6}\right)^{m/2} \]  

This proves Proposition \ref{momentsk=1}.
\begin{remark}
If $\alpha$ and $\beta$ are not irrationals linearly independent over $\mathbb{Q}$, Wieand \cite{wieand} has also calculated the limit of the quadratic sums ($m=2$ in Lemma \ref{calculation}) for various cases.  A modification of this moment method then gives that the limiting distribution is a normal distribution with variance given by the limit of the quadratic sum. 
\end{remark}

\begin{remark} \label{Znstar}
Let $W_j$ be independent Poisson variables with parameter $\theta/j$.  For each $n$, define \begin{equation} Z_n^* := \frac{1}{\sqrt{\log n}} \sum_{j=1}^n W_j (\{j\alpha \} - \{j\beta \}) 
\end{equation} 
It is easy to see that Lemma \ref{Znlimitlemma}, and therefore Lemma \ref{Znstirlinglemma} hold if we replace the random variables $C_{j}^{(n)}$ with $W_j$ and therefore all moments of $Z_n$ and $Z_n^*$ have the same limit as $n \to \infty$.  The method of moments thus gives an alternative to the Feller coupling method of seeing that $Z_n$ and $Z_n^*$ converge to the same limit (since the normal distribution is characterized by its moments).  
\end{remark}

\section{Moment method for $k > 1$} \label{k>1}

In this section, we prove Theorem \ref{momentsthm}.  With the same notation as in the previous section, we have

\begin{multline} \label{Ymultinomial}
\big(Y_{n,k}\big)^m = \frac{1}{n^{m(k-1)}} \sum_{(m_i) \vdash m} \binom{m}{ m_1,...,m_s} \frac{1}{\prod_{l = 1}^m c_l!} \\
\sum_{j_1 \neq ... \neq j_s} \prod_{l=1}^s j_l^{m_l(k-1)} \big(C_{j_l}^{(n)}\big)^{m_l} (\{j_l\alpha \} - \{j_l\beta \})^{m_l} 
\end{multline}

First, we prove the following Riemann sum approximation:  
\begin{lemma} \label{Yndeltaboundlemma}
Let $m = m_1+...+m_s$ be a partition of $m$ and $r = r_1+...+r_s$ where $1 \le r_l \le m_l$.  As usual, all indices $1 \le j_1,...,j_s \le n$.  Then

\begin{align} \label{Yndeltabound}
& \nonumber \lim_{n \to \infty} \frac{1}{n^{m(k-1) +s-r}} \sum_{1 \le \sum j_l r_l \le n} \prod_{l=1}^s j_l^{m_l(k-1) - r_l} \prod_{i=0}^{\sum j_l r_l -1} \frac{n-i}{\theta+n-i-1} \\
= & \int \limits_{0  \le \sum x_l r_l \le 1}  \prod_{l=1}^s x_l^{m_l(k-1) - r_l} \left( \frac{1}{1-\sum x_l r_l} \right)^{1-\theta} dx_1...dx_s
\end{align}

\end{lemma}

\begin{proof}
By \eqref{thetabound},
\begin{align}
&\nonumber \frac{1}{n^{m(k-1) +s-r}} \sum_{1 \le \sum j_l r_l \le n} \prod_{l=1}^s j_l^{m_l(k-1) - r_l} \prod_{i=0}^{\sum j_l r_l -1} \frac{n-i}{\theta+n-i-1} \\
\ll &\nonumber \frac{1}{n^{m(k-1) +s-r}} \sum_{1 \le \sum j_l r_l \le n} \prod_{l=1}^s j_l^{m_l(k-1) - r_l} \left( \frac{\theta+n-1}{\theta+n-\sum j_l r_l} \right)^{1-\theta} \\
\sim &\nonumber \frac{1}{n^s} \sum_{1 \le \sum j_l r_l \le n} \prod_{l=1}^s \Big(\frac{j_l}{n}\Big)^{m_l(k-1) - r_l} \left( \frac{1}{1-\sum j_l r_l/n} \right)^{1-\theta} \\
\sim & \int \limits_{0 \le \sum x_l r_l \le 1}  \prod_{l=1}^s x_l^{m_l(k-1) - r_l} \left( \frac{1}{1-\sum x_l r_l} \right)^{1-\theta} dx_1...dx_s
\end{align}
One can check that this integral is finite for $\theta > 0$.  By continuity of the integral, and \eqref{thetasim}, we can then replace the first $\ll$ with $\sim$ and the result follows.
\end{proof}

The following lemma computes the inner sum in \eqref{Ymultinomial}:
\begin{lemma} \label{Ynlimitlemma}
Let $m = m_1+...+m_s$ be a partition of $m$.  Then
\begin{multline} \label{Ynlimit}
\lim_{n \to \infty} \frac{1}{n^{m(k-1)}}  \sum_{j_1 \neq ... \neq j_s} \prod_{l=1}^s j_l^{m_l(k-1)} \mathbb{E}\Big[\prod_{l=1}^s C_{j_l}^{(n)} \Big] (\{j_l\alpha \} - \{j_l\beta \})^{m_l} \\ = \lim_{n \to \infty} \frac{1}{n^{m(k-1)}} \sum_{j_1+...+j_s\le n} \left(\frac{n}{n-(j_1+...+j_s)} \right)^{1-\theta} \prod_{l=1}^s \big( \theta j_l^{m_l(k-1) - 1} \big) (\{j_l\alpha \} - \{j_l\beta \})^{m_l}
\end{multline}

\end{lemma}
\begin{proof}
Let $r_1,...,r_s$ be integers such that $1 \le r_l \le m_l$.  Then by \eqref{diaconispoisson} and Lemma \ref{Yndeltaboundlemma},
\begin{multline*}
\lim_{n \to \infty} \frac{1}{n^{m(k-1)}}  \sum_{j_1 \neq ... \neq j_s} \prod_{l=1}^s j_l^{m_l(k-1)} \mathbb{E}\Big[\prod_{l=1}^s \Big(C_{j_l}^{(n)}\Big)^{\underline{r_l}} \Big] (\{j_l\alpha \} - \{j_l\beta \})^{m_l} \\ = \lim_{n \to \infty} \frac{1}{n^{m(k-1)}} \sum_{j_1+...+j_s\le n} \left(\frac{n}{n-(j_1+...+j_s)} \right)^{1-\theta} \prod_{l=1}^s \big( \theta j_l^{m_l(k-1) - 1} \big) (\{j_l\alpha \} - \{j_l\beta \})^{m_l}
\end{multline*}
if $r_l=1$ for all $l$ and otherwise the limit is 0.  Equation \eqref{Ynlimit} then follows from formula \eqref{stirling}.

\end{proof}

Wieand \cite{wieand} uses Exercise 65 from Szeg\H o and P\' olya's \textit{Problems and Theorems in Analysis I} \cite{polya} to prove \eqref{logarithmic}.  To compute the limit in \eqref{Ynlimit}, we need a slightly modified version of this exercise.  

\begin{lemma}\label{szego} 
Define the following data:

Let $f(n)$ and $g_\delta(n)$ be increasing functions for each $\delta > 0$.  Let $s_{n i}$, $1 \le i \le f(n)$ be a bounded array such that for each $\delta > 0$, there exists some limiting value $L$ such that \[\lim_{n \to \infty} \max_{i > g_\delta(n)} |s_{n i} - L| = 0 \]

Also, let $p_{n i}$, $1 \le i \le f(n)$ be an array such that $\displaystyle{\sum_{i=1}^{f(n)} p_{n i} = 1}$, $\displaystyle{\sum_{i=1}^{f(n)} |p_{n i}|}$ is bounded, and such that $\displaystyle{\lim_{\delta \to 0} \limsup_{n \to \infty}\sum_{i=1}^{g_\delta(n)} |p_{n i}| = 0}$.  Finally, let $\displaystyle{t_n = \sum_{i=1}^{f(n)} p_{n i} s_{n i}}$.  Then $\displaystyle{\lim_{n \to \infty} t_n = L}$.
\end{lemma}
\begin{proof}
The proof is a simple modification of the argument in Szeg\H o and P\' olya's exercise. \end{proof}
 
We now apply this lemma to the sum in \eqref{Ynlimit}.  First, consider the simplest case $s=1$ and $\theta=1$.
\begin{lemma}\label{s=1} If $m$ is even (and positive), \begin{equation} \lim_{n \to \infty} \frac{1}{n^{m(k-1)}} \sum_{j=1}^n j^{m(k-1)-1} (\{j\alpha \} - \{j\beta \})^m = \frac{2}{(k-1)m(m+1)(m+2)}\end{equation}  If $m$ is odd, the limit is 0.
\end{lemma}
\begin{proof}
Set \[s_{n i} = s_i := \frac{1}{i} \sum_{j=1}^i w_j \] and \[t_n = \frac{1}{\sum_{j=1}^n b_j} \sum_{j=1}^n b_j w_j \] where $b_j = j^{m(k-1)-1}$ and  $w_j = (\{j\alpha \} - \{j\beta \})^m$.
A simple calculation (essentially Abel summation) shows that $t_n = \sum \limits_{i=1}^n p_{n i} s_i$ where $\displaystyle{p_{n i} = \frac{i(b_i - b_{i+1})}{\sum_{j=1}^n b_j}}$ when $i < n$ and $\displaystyle{p_{nn} = \frac{n b_n}{\sum_{j=1}^n b_j}}$.  Note that all the terms $p_{n i}$ are negative except $p_{n n}$.  Setting $f(n) = n$ and $g_\delta(n) = \delta n$, we see that all the conditions for the array $p_{ni}$ in Lemma \ref{szego} are satisfied.  Then \[\lim_{n \to \infty} \frac{m(k-1)}{n^{m(k-1)}} \sum_{j=1}^n j^{m(k-1)-1} (\{j\alpha \} - \{j\beta \})^m = \lim_{n \to \infty} t_n = \lim_{n \to \infty} \frac{1}{n} \sum_{j=1}^n (\{j\alpha \} - \{j\beta \})^m \]
The result follows by Lemma \ref{calculation}.
\end{proof}

More generally, we need to deal with a multi-dimensional sequence.

\begin{lemma}\label{generalmainlemma}
Define a partition $\displaystyle{m = m_1+...+m_s}$ such that $\displaystyle{m_1 \ge... \ge m_s \ge 1}$.  If all $m_l$ are even, 

\begin{multline}
\lim_{n \to \infty} \frac{1}{n^{m(k-1)}} \sum_{j_1+...+j_s\le n} \left(\frac{n-(j_1+...+j_s)}{n} \right)^{\theta-1} \prod_{l=1}^s \big( \theta j_l^{m_l(k-1) - 1} \big) (\{j_l\alpha \} - \{j_l\beta \})^{m_l} \\ = \frac{\Gamma(\theta)}{\Gamma(m(k-1)+\theta)} \prod_{l=1}^s \frac{2 \theta \Gamma(m_l(k-1))}{(m_l+1)(m_l+2)}
\end{multline}
where $\Gamma(z)$ is the Gamma function.  Otherwise, the limit is 0.

\end{lemma}

\begin{proof}
Define the set of integer lattice points \begin{equation} \mathfrak{X}_n = \{ (j_1,...,j_s) \in [n]^s:  \sum_{l=1}^s j_l \le n \} \end{equation} and set $\displaystyle{f(n) = |\mathfrak{X}_n|}$.  Let $b_{n, i}$ denote the elements of the multiset \[\bigg\{  \left(\frac{n-(j_1+...+j_s)}{n} \right)^{\theta-1} \prod_{l=1}^s \big( \theta j_l^{m_l(k-1) - 1} \big) : (j_1,...,j_s) \in \mathfrak{X}_n \bigg\} \] listed in increasing order.  This induces a (not necessarily unique) ordering on the underlying set $\mathfrak{X}_n$.  Let $\mathfrak{X}_n^i$ be the set consisting of the first $i$ elements of $\mathfrak{X}_n$ under this ordering. 

Following the proof of Lemma \ref{s=1}, we define the arrays as follows.  
Set \begin{equation*}\label{sn} s_{n i} = \frac{1}{i}\sum_{(j_1,...,j_s) \in \mathfrak{X}_n^i} \prod_{l=1}^s (\{j_l\alpha \} - \{j_l\beta \})^{m_l} 
\end{equation*} and \begin{equation*} t_n = \frac{1}{\sum_{i=1}^{f(n)} b_{n,i}} \sum_{(j_1,...,j_s) \in \mathfrak{X}_n}  \left(\frac{n-(j_1+...+j_s)}{n} \right)^{\theta-1} \prod_{l=1}^s \big( \theta j_l^{m_l(k-1) - 1} \big) (\{j_l\alpha \} - \{j_l\beta \})^{m_l} \end{equation*}

Then $\displaystyle{t_n = \sum_{i=1}^{f(n)} p_{ni} s_{ni}}$ where $\displaystyle{p_{n i} = \frac{i(b_{n, i} - b_{n, i+1})}{\sum_{j=1}^{f(n)} b_{n, j} }}$ for $i < f(n)$ and $\displaystyle{p_{n, f(n)} = \frac{f(n) b_{n, f(n)}}{\sum_{j=1}^{f(n)} b_{n, j} }}$.   Clearly $\displaystyle{\sum_{i=1}^{f(n)} p_{ni} = 1}$ and since $b_{n,i}$ is a nondecreasing sequence, $\displaystyle{\sum_{i=1}^{f(n)} |p_{n i}|}$ will be bounded.   

By Riemann integral comparison, 
\begin{align} 
\lim_{n \to \infty}\frac{1}{n^{m(k-1)}} \sum_{j=1}^{f(n)} b_{n, j} &\nonumber= \lim_{n \to \infty} \sum_{\sum_{l=1}^s j_l \le n} \frac{\theta^s}{n^s} \big(1 - (j_1+...+j_s)/n\big)^{\theta-1} \prod_{l=1}^s \left(\frac{j_l}{n}\right)^{m_l(k-1)-1} \\ &\nonumber= \int \limits_{\sum_{l=1}^s x_l \le 1} \theta^s (1- x_1-...-x_s)^{\theta-1} \prod_{l=1}^s x_l^{m_l(k-1)-1} dx_l \\ &= \label{mathematica} \theta^s \frac{\Gamma(\theta)}{\Gamma(m(k-1)+\theta)} \prod_{l=1}^s \Gamma(m_l(k-1)) 
\end{align}
where the last equality represents the normalizing constant for the Dirichlet distribution with parameters $m_1(k-1),...,m_s(k-1), \theta > 0$.
For each $\delta > 0$, define the set \begin{equation} \label{Yg} \mathfrak{Y}_{\delta, n} = \Big\{ (j_1,...,j_s) \in \mathfrak{X}_n:  \left(\frac{n-(j_1+...+j_s)}{n} \right)^{\theta-1} \prod_{l=1}^s \big( \theta j_l^{m_l(k-1) - 1} \big) \le \delta n^{m(k-1)-s} \Big\} 
\end{equation} and let $g_\delta(n) = |\mathfrak{Y}_{\delta, n}|$.  
Note that by scaling $\mathfrak{Y}_{\delta, n}$ by a factor of $n$, the asymptotics of $g_\delta(n)$ can also be computed via comparison to an integral (volume approximation).
\begin{equation} \label{gasymp} \lim_{n \to \infty} \frac{g_\delta(n)}{n^s} = \int \limits_{\substack{ \sum x_l \le 1 \cr (1-\sum x_l)^{\theta-1} \prod \theta x_l^{m_l(k-1)-1} \le \delta } } dx_1...dx_s
\end{equation}
Then 
\begin{equation*} \begin{aligned}\lim_{\delta \to 0} \limsup_{n \to \infty} \frac{1}{n^{m(k-1)}} \sum_{i=1}^{g_\delta(n)} |i(b_{n, i} - b_{n, i+1})| &\le \lim_{ \delta \to 0} \limsup_{n \to \infty} \frac{2}{n^{m(k-1)}} g_\delta(n) \delta n^{m(k-1)-s}  \\ &= \lim_{\delta \to 0} 2 \delta \int \limits_{\substack{ \sum x_l \le 1 \cr (1-\sum x_l)^{\theta-1} \prod \theta x_l^{m_l(k-1)-1} \le \delta } } dx_1...dx_s = 0 
\end{aligned}
\end{equation*}

It remains to show that for each $\delta > 0$, there exists some limiting value $L$ such that \[\lim_{n \to \infty} \max_{i > g_\delta(n)} |s_{n i} - L| = 0\]  Define the mapping $\phi:[n]^s \to \mathbb{T}^{2s}$ given by \begin{equation} \phi(j_1,...,j_s) = (j_1 \alpha, j_1 \beta ,...,j_s \alpha , j_s \beta  )
\end{equation}  
Lemma \ref{tupledisc} below shows that $\displaystyle{ \lim_{n \to \infty} \max_{i > g_\delta(n)} |D(\phi(\mathfrak{X}_n^i))| = 0 }.$  This means that the condition on the array $s_{ni}$ is satisfied with $\displaystyle{ L := \int_0^1...\int_0^1 \prod_{l=1}^s (x_l - y_l)^{m_l}dx_ldy_l }$.  By Lemma \ref{calculation}, $L = \newline \displaystyle{\prod_{l=1}^s \frac{2}{(m_l+1)(m_l+2)}}$ if all the exponents $m_l$ are even and $L = 0$ otherwise.  The result follows by combining this with \eqref{mathematica}.
    
\end{proof}

To finish the proof of Lemma \ref{generalmainlemma}, we show:

\begin{lemma}\label{tupledisc}

Following the notation from the proof of Lemma \ref{generalmainlemma}, for each $\delta > 0$, \begin{equation}\label{lemmadisc} \lim_{n \to \infty} \max_{i > g_\delta(n)} |D(\phi(\mathfrak{X}_n^i))| = 0 \end{equation}  
\end{lemma}

\begin{proof}
For any $A > 0$, the lattice $(A \mathbb{Z})^s$ defines a partition of $\mathbb{R}_+^s$ into $s$-dimensional cubes $C_{i,A}$ of side length $A$.  Define the cubes so that the boundaries do not overlap and order them according to distance from the origin to the center of the cube.  For each $\varepsilon > 0$ (and $A > 1$), this then induces a partition $\displaystyle{ \mathfrak{Y}_{\varepsilon, n} = \bigcup_{i} (C_{i,A} \cap \mathfrak{Y}_{\varepsilon, n} ) }$ where $\mathfrak{Y}_{\varepsilon, n}$ is defined in \eqref{Yg}.  Given $\delta > 0$, we need to show that $D(\phi(\mathfrak{Y}_{\varepsilon, n})) \to 0$ as $n \to \infty$ uniformly over $\varepsilon \ge \delta$.  

Since $(j \alpha, j \beta)$ is well-distributed, the discrepancies over the cubes tend to 0 uniformly as the side length approaches infinity, i.e. $\displaystyle{ \lim_{A \to \infty} \sup_i D(\phi(C_{i,A} \cap \mathbb{Z}^s)) = 0 }$.  It is not hard to see (e.g. \cite[Thm 2.6]{kuipers}) that \begin{equation}\label{di} D(\phi(\mathfrak{Y}_{\varepsilon, n})) \le  \sum_i \frac{|C_{i,A} \cap \mathfrak{Y}_{\varepsilon, n}|}{ |\mathfrak{Y}_{\varepsilon, n}| } D(\phi(C_{i,A} \cap \mathfrak{Y}_{\varepsilon, n})) 
\end{equation}   
Define the set of indices $S_{A, \varepsilon, n}$ such that $(C_{i,A} \cap \mathbb{Z}^s) \neq (C_{i,A} \cap \mathfrak{Y}_{\varepsilon, n})$ for $i \in S_{A, \varepsilon, n}$.  In words, $\{C_{i,A} : i \in S_{A, \varepsilon, n} \}$ is the set of cubes on the boundary of $\mathfrak{Y}_{\varepsilon, n}$.  Then by \eqref{di}, it suffices to show that for each $A > 0$, \begin{equation}\label{boundary} \sum_{i \in S_{A, \varepsilon, n}} \frac{|C_{i,A} \cap \mathfrak{Y}_{\varepsilon, n}|}{ |\mathfrak{Y}_{\varepsilon, n}| } \to  0 \end{equation} as $n \to \infty$ uniformly over $\varepsilon \ge \delta$.  Shrink $\mathbb{Z}^s$ by a factor of $n$, which induces a corresponding scaling of the subsets $\mathfrak{Y}_{\varepsilon, n}$ and $C_{i,A} \cap \mathfrak{Y}_{\varepsilon, n}$.  Then as in \eqref{gasymp}, \eqref{boundary} follows by volume approximation.  The discrepancy bound \eqref{lemmadisc} follows from \eqref{di} by taking $A \to \infty$. 
\end{proof}

Putting this together, we see that for even $m$,

\begin{align}
\mathbb{E}[\big(Y_{\infty, k}\big)^m] &\nonumber= \lim_{n \to \infty} \mathbb{E}[\big(Y_{n,k} \big)^m] \\ &\nonumber= \sum_{\substack{(m_i) \vdash m \\ m_i \text{ even}}} \binom{m}{ m_1,...,m_s} \frac{1}{\prod_{l = 1}^m c_l!} \frac{\Gamma(\theta)}{\Gamma(m(k-1)+\theta)} \prod_{l=1}^s \frac{2 \theta \Gamma(m_l(k-1))}{(m_l+1)(m_l+2)} \\
&= \label{momenttts} \frac{\Gamma(\theta)}{\Gamma(m(k-1)+\theta)} \sum_{\substack{(m_i) \vdash m \\ m_i \text{ even}}} \frac{m!}{\prod_{l = 1}^m c_l!}\prod_{l=1}^s \frac{2 \theta \Gamma(m_l(k-1))}{(m_l+2)!} 
\end{align} 

and $\displaystyle{\mathbb{E}[\big(Y_{\infty, k}\big)^m] = 0}$ for odd $m$.

\begin{definition}
The partial Bell polynomials are given by \[B_{n,k}(x_1,...,x_{n-k+1}) = \sum \frac{n!}{j_1!...j_{n-k+1}!} \left(\frac{x_1}{1!} \right)^{j_1}...\left(\frac{x_{n-k+1}}{(n-k+1)!} \right)^{j_{n-k+1}}\]
where the sum is taken over all sequences $j_1,...,j_{n-k+1}$ of non-negative integers such that $j_1+...+j_{n-k+1} = k$ and $j_1+2j_2+...+(n-k+1) j_{n-k+1} = n$.  Then the complete Bell polynomials are defined by $\displaystyle{B_n(x_1,...,x_n) = \sum_{k=1}^n B_{n,k}(x_1,...,x_{n-k+1})}$.  
\end{definition}
The Bell polynomials satisfy the exponential formula: \begin{equation} \exp \left( \sum_{n=1}^\infty \frac{a_n}{n!} x^n \right) = \sum_{n=0}^\infty \frac{B_n(a_1,...,a_n)}{n!} x^n 
\end{equation}

Using this formula, one sees that as formal power series, \begin{equation*} \sum_{m = 0}^\infty \mathbb{E}[(Y_{\infty, k})^m] \frac{\Gamma(m(k-1)+\theta)}{\Gamma(\theta) m!} z^m = \exp(K(z))
\end{equation*} where $\displaystyle{K(z) = \sum_{m=1}^\infty \kappa_{2m} z^{2m}}$ and $\displaystyle{\kappa_{2m} = \frac{2\theta\Gamma(2m(k-1))}{(2m+2)!} }$.  This completes the proof of Theorem \ref{momentsthm}.

\section{Eigenvalue density when $k = 2$ and $\theta=1$} \label{k=2}

We will now specialize to the case $k = 2$ and $\theta=1$.  Then $\displaystyle{\kappa_{2m} = \frac{2}{2m(2m+1)(2m+2)}}$ and $K(z)$ has radius of convergence 1.  Let us determine a closed form for $K(z)$ in this region.  Note that $\displaystyle{(z^2 K(z))''' = \frac{2z}{1-z^2}}$.  Solving this differential equation, we find that \begin{equation} K(z) = \frac{3}{2} - \frac{1}{2}\left(1 - \frac{1}{z}\right)^2 \log(1 - z) - \frac{1}{2}\left(1 + \frac{1}{z}\right)^2 \log(1 + z) \end{equation}
which is well defined for $|z| < 1$.  By taking the branch cut of $\log(1-z)$ to be $[1, \infty)$ and the branch cut of $\log(1+z)$ to be $(-\infty, -1]$, we see that $K(z)$ can be extended analytically to $\mathbb{C} \setminus \{ (-\infty, -1] \cup [1, \infty) \}$.  To extract the density, we use the Stieltjes transform and the associated inversion formula.      

\begin{definition}
For a probability measure $\mu$, the Stieltjes transform is given by \[G_\mu(z) = \int_{\mathbb{R}} \frac{1}{z-t } \mu(dt)\]  It is well-defined on $\displaystyle{\mathbb{C} \setminus \text{support}( \mu)}$.  The Stieltjes inversion formula states \[ d\mu(x) = \lim_{\varepsilon \to 0^+} \frac{G_\mu(x - i \varepsilon) - G_\mu(x + i \varepsilon)}{2 i \pi}\]  In particular, if this limit exists for all $x$ in the support of $\mu$, the formula gives the continuous density function $\rho$ of $\mu$.
\end{definition}

As formal power series, the Stieltjes transform $\displaystyle{G(z) := \mathbb{E}\bigg[\frac{1}{z - Y_{\infty, 2}}\bigg]}$ equals $\displaystyle{\frac{1}{z} \exp(K(1/z))}$.  This must also be the asymptotic series expansion of $G(z)$ as $z \to \infty$ since it is determined uniquely.  Thus, we have the equality $\displaystyle{G(z) = \frac{1}{z} \exp(K(1/z))}$ as analytic functions for $|z| > 1$.  By the uniqueness of analytic continuation, this equality holds true in fact for $z \in \mathbb{C} \setminus [-1, 1]$.  

Using the inversion formula on the random variable $Y_{\infty, 2}$, we get for $-1 < t < 0$, 
\begin{align*}
& \lim_{\epsilon \to 0^+} \Im G(t + i \epsilon) \\
= & \frac{1}{t} \exp\left(\frac{3}{2} - \frac{1}{2} (1 - t)^2 \log\Big(1 - \frac{1}{t}\Big) \right) 
\lim_{\epsilon \to 0^+} \Im \exp\left(-\frac{1}{2} (1+t)^2 \log \Big(1 + \frac{1}{t + i \epsilon}\Big)\right) \\
= & \frac{1}{t} \exp\left(\frac{3}{2} - \frac{1}{2} (1 - t)^2 \log \Big(1 - \frac{1}{t} \Big)\right) \Im \exp\left(-\frac{1}{2} (1+t)^2 \Big(\log \Big(\frac{1}{|t|} - 1 \Big) - i \pi \Big)\right) \\
= & \frac{1}{t} \exp\left(\frac{3}{2} - \frac{1}{2} (1 - t)^2 \log \Big(1 - \frac{1}{t} \Big) -\frac{1}{2} (1+t)^2 \log \Big(\frac{1}{|t|} - 1 \Big)\right)\sin \bigg(\frac{(1+t)^2}{2} \pi \bigg)
\end{align*}
Thus, by symmetry (since all odd moments of $Y_{\infty, 2}$ are zero) the density is
\begin{align}
&\nonumber p_{Y_{\infty, 2}}(t) \\
= & -\frac{1}{\pi} \lim_{\epsilon \to 0^+} \Im G(-|t| + i \epsilon) \nonumber \\
= &\frac{1}{\pi |t|} \exp\bigg(\frac{3}{2} - \frac{1}{2} (1 + |t|)^2 \log\Big(1 + \frac{1}{|t|}\Big) -\frac{1}{2} (1-|t|)^2 \log\Big(\frac{1}{|t|} - 1 \Big)\bigg) \sin \bigg(\frac{(1-|t|)^2}{2} \pi \bigg) \nonumber \\
= & \frac{e^{3/2}}{\pi |t|} \left(\frac{1}{|t|} - 1\right)^{-\frac{1}{2} (1 - |t|)^2} \left(\frac{1}{|t|} + 1\right)^{-\frac{1}{2} (1 + |t|)^2} \sin \bigg(\frac{(1-|t|)^2}{2} \pi \bigg) 
\end{align}
\vspace{12pt}
for $-1 \le t \le 1$ which proves Corollary \ref{twocor}.  
 
\begin{figure}
\centering
\includegraphics[width=70mm]{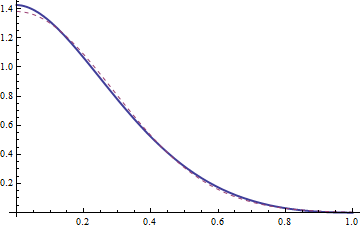}
\caption{solid line = $\displaystyle{p_{Y_{\infty, 2}}(t)}$, dotted line = normal density}
 \label{densityfigure}
\end{figure} 
 
Figure \ref{densityfigure} shows the graphs of $\displaystyle{p_{Y_{\infty, 2}}(t)}$ and a normal density with mean 0 and variance $1/12$.  It is striking how similar the two densities look. (However, the Gaussian density is of course not compactly supported.) 
  
\begin{remark} 
Let $\alpha_1, \alpha_2, \beta_1, \beta_2$ be irrational numbers linearly independent over $\mathbb{Q}$ and $I_1 = (\alpha_1, \beta_1)$ and $I_2 = (\alpha_2, \beta_2)$ be two intervals.  Wieand \cite{wieand} showed that for the defining representation on $\mathfrak{S}_n$ with uniform measure, the normalized eigenvalue counts $Z^{I_1}_n$ and $Z^{I_2}_n$ converge in distribution to independent normal random variables.  This is because $\lim_{n \to \infty} Cov(Z^{I_1}_n, Z^{I_2}_n) = 0$ and the multivariate normal distributions are determined by their covariance structure.  

However, by computing cross moments, one sees that for the $k=2$ representations on $\mathfrak{S}_n$ with uniform measure, $Y^{I_1}_{n,2}$ and $Y^{I_2}_{n,2}$ do not converge to independent random variables.  For example, a little calculation shows that 
\begin{align}
\lim_{n \to \infty} \mathbb{E}[(Y^{I_1}_{n,2})^2] \mathbb{E}[(Y^{I_2}_{n,2})^2] &= \frac{2}{2 \cdot 3 \cdot 4} \frac{2}{2 \cdot 3 \cdot 4} \\
\lim_{n \to \infty} \mathbb{E}[(Y^{I_1}_{n,2})^2 (Y^{I_2}_{n,2})^2] &= \frac{1}{4} \frac{2}{3 \cdot 4} \frac{2}{3 \cdot 4} + \frac{1! 1!}{4!}\frac{2}{3 \cdot 4}\frac{2}{3 \cdot 4}
\end{align}

Thus, unlike the $k = 1$ case, the squares of the random variables $Y^{I_1}_{n,2}$ and $Y^{I_2}_{n,2}$ are positively correlated in the limit.  
\end{remark}

\section{General linear combinations of cycle lengths} \label{generalf}

We now put our results in the context of the prior work of Ben Arous and Dang \cite{arousdang}.  Let $(u_j)_{j \ge 1}$ be a sequence of real numbers and let the random variable $\displaystyle{ X_n^{(u_j)} := \sum_{j=1}^n u_j C_j^{(n)} }$ be the associated linear combination of cycle lengths.  Ben Arous and Dang obtain two different limiting laws for $X_n^{(u_j)}$ depending on the conditions that the sequence $(u_j)_{j \ge 1}$  satisfies.  Theorem \ref{smooth} below contains parts (1) and (2) of Theorem 2.3 in \cite{arousdang}.  Theorem \ref{rough} below is part (1) of Theorem 2.4 in \cite{arousdang}.  See Definition 2.1 in \cite{arousdang} for the definition of convergence in the Cesaro $(C,\theta)$ sense.

\begin{theorem} \label{smooth}
Let $\theta > 0$ and assume that $\displaystyle{\sum_{j=1}^\infty \frac{u_j^2}{j} \in (0, \infty) }$.  If $0 < \theta < 1$, assume additionally that the sequence $(|u_j|)_{j \ge 1}$ converges to zero in the Cesaro $(C, \theta)$ sense.  Then under the Ewens distribution with parameter $\theta$, $\displaystyle{X_n^{(u_j)} - \mathbb{E}[X_n^{(u_j)}] } $ converges weakly as $n \to \infty$ to a non-Gaussian infinitely divisible distribution defined by its Fourier transform \[\phi(t) = \exp \left(\theta \int(e^{itx} - 1 - itx) dM_f(x) \right) \] where the L\'{e}vy measure $M_f$ is given by $\displaystyle{ M_f = \sum_{j=1}^\infty \frac{1}{j} \delta_{u_j} }$.  
\end{theorem}

\begin{theorem} \label{rough}
Let $\theta > 0$ and assume that $\displaystyle{\sum_{j=1}^\infty \frac{u_j^2}{j} = \infty }$ and that $\max \limits_{1 \le j \le n} |u_j| = o(\eta_n)$ where $\displaystyle{ \eta_n^2 = \theta \sum_{j=1}^n \frac{u_j^2}{j} }$.  Then under the Ewens distribution with parameter $\theta$, the centered and normalized eigenvalue statistic $\displaystyle{ \frac{X_n^{(u_j)} - \mathbb{E}[X_n^{(u_j)}]}{\sqrt{\Var X_n^{(u_j)}}} }$ converges weakly as $n \to \infty$ to the standard normal $\mathcal{N}(0, 1)$.  
\end{theorem}

As in \cite{wieand}, the main thrust of the proofs of these two theorems is the Feller coupling that relates cycle lengths $C_j^{(n)}$ to independent Poisson variables $W_j$ with parameter $\theta/j$.  If the sequence $(u_j)_{j \ge 1}$ satisfies the hypotheses of Theorem \ref{smooth} or \ref{rough}, then $X_n^{(u_j)}$ and $\displaystyle{ X_n^{*, (u_j)} := \sum_{j=1}^n u_j W_j }$ will have the same limiting behavior.  It is then easy to compute the limiting law of (the normalized version of) $\displaystyle{ X_n^{*, (u_j)}}$ and see that it is infinitely divisible and given either by Theorem \ref{smooth} or \ref{rough} depending on the asymptotics of $(u_j)$.  

Note that the random variables $Y_{n,k}$ studied in this work correspond to $X_n^{(u_j)}$ where $u_j = j^{k-1} (\{j \alpha \} - \{j \beta \})$.  As shown in Theorem \ref{momentsthm}, the limiting distribution $Y_{\infty, k}$ is compactly supported for $k>1$, hence not infinitely divisible.  Thus, we've uncovered a new class of limiting distributions $Y_{\infty, k}$ not present in \cite{arousdang}.  The random variables $Y_{n,k}$ of course must fail to satisfy the hypotheses of both Theorems \ref{smooth} and \ref{rough} and indeed $\displaystyle{\lim_{n \to \infty} \eta_n = \infty}$ and $\max \limits_{1 \le j \le n} |u_j| = \Omega(\eta_n)$ where $\eta_n$ is defined as in Theorem \ref{rough}.  Let $Y_{n,k}^*$ be the normalized version of $X_n^{*,(u_j)}$, i.e.
\begin{equation} Y_{n,k}^* := \sum_{j=1}^n \frac{j^{k-1} W_j (\{j \alpha\} - \{j \beta\})}{n^{k-1} } 
\end{equation} 
Using the L\'{e}vy-Khintchine Representation Theorem, it is easy to state the limiting distribution of $Y_{n,k}^*$.  Recall that Kolmogorov's theorem \cite[p. 162]{durrett}, a special case of the L\'{e}vy-Khintchine Theorem, states that a random variable $Z$ has an infinitely divisible distribution with mean 0 and finite variance if and only if its characteristic function has \[ \log \phi(t) = \int (e^{it x} - itx - 1)x^{-2} \nu(dx) \] where $\nu$ is called the canonical measure and $\Var Z = \nu(\mathbb{R})$.  

\begin{proposition} \label{Ynkstar}
The random variables $Y_{n,k}^*$ converge weakly to a random variable $Y_{\infty,k}^*$ with an infinitely divisible law given by \[\log \mathbb{E}[\exp(i t Y_{\infty,k}^*)] = \int_{-1}^1 (e^{i t x} - it x - 1)x^{-2}\nu(dx) \] where the canonical measure is supported on the interval $[-1, 1]$ and given by \[\nu(dx) = \frac{\theta}{k-1} \left(-x^2 + \frac{|x|^3}{2} + \frac{|x|}{2} \right) dx\] with $\displaystyle{\nu(\mathbb{R}) = \frac{\theta}{12(k-1)} }$.
\end{proposition}

\begin{proof}
Let $\displaystyle{ a_j = \frac{j^{k-1}(\{j \alpha\} - \{j \beta\})}{n^{k-1}} }$.  We have \begin{equation} \label{fininfdiv}
 \log \mathbb{E}[\exp(i t Y_{n,k}^*)] = \log \prod_{j=1}^n \mathbb{E}[\exp(i t a_j W_j)] = \sum_{j=1}^n \frac{\theta}{j} (e^{i t a_j} - 1 ) 
\end{equation}
Then 
\begin{align}
\nonumber \lim_{n \to \infty} \log \mathbb{E}[\exp(i t Y_{n,k}^*)] &= \lim_{n \to \infty} \sum_{j=1}^n \frac{\theta}{j} \big(\exp \bigg(i t \frac{j^{k-1}(\{j \alpha\} - \{j \beta\})}{n^{k-1}}  \bigg) - 1  \big) \\
& \nonumber = \lim_{n \to \infty} \sum_{m=1}^\infty \frac{(it)^m}{m!} \sum_{j=1}^n \frac{\theta}{j} \bigg(\frac{j^{k-1}}{n^{k-1}} \bigg)^m ( \{j \alpha\} - \{j \beta \})^m \\
& \label{infdiv} = \sum_{m=1}^\infty \frac{(it)^{2m}}{(2m)!} \frac{2\theta}{(k-1)2m(2m+1)(2m+2)} 
\end{align}
where the last equality is by Lemma \ref{s=1}.  By \eqref{fininfdiv}, we see that $Y_{n,k}^*$ has an infinitely divisible distribution for each $n$ and one can check that \eqref{infdiv} can be written as the integral
\begin{equation}
\frac{\theta}{k-1}\int_{-1}^1 (e^{i t x} - it x - 1)\left(-1 + \frac{|x|}{2} + \frac{1}{2|x|} \right) dx
\end{equation}

\end{proof}

\begin{remark}
For each $t \ge 0$, let $\mu_t$ denote the infinitely divisible law with canonical measure in Kolmogorov's theorem supported on the interval $[-1, 1]$ and given by \[\nu(dx) = t \left(-x^2 + \frac{|x|^3}{2} + \frac{|x|}{2} \right) dx .\]  Then note that the law forms a semigroup, i.e. $\mu_s \ast \mu_t = \mu_{s+t}.$  $\mu_0 = \delta_0$.  Thus, $\mu_t$ is a L\'{e}vy process.  The law of $Y_{\infty, k}^*$ is $\mu_{\theta/(k-1)}$.  This shows how they all fit into the same L\'{e}vy process.
\end{remark}

Proposition \ref{Ynkstar} shows that $Y_{n,k}^*$ converges to an infinitely divisible law and hence $Y_{n,k}$ and $Y_{n,k}^*$ do not have the same limiting distribution.  Thus, the Feller coupling between $C_j^{(n)}$ and $W_j$ breaks down here.  One way to see how the difference arises is again from the moment method.  

For the Poisson variable sum, instead of \eqref{Ynlimit}, we have
\begin{multline} \label{Ynstarlimit}
\lim_{n \to \infty} \frac{1}{n^{m(k-1)}}  \sum_{j_1 \neq ... \neq j_s} \prod_{l=1}^s j_l^{m_l(k-1)} \mathbb{E}\Big[\prod_{l=1}^s W_{j_l} \Big] (\{j_l\alpha \} - \{j_l\beta \})^{m_l} \\ = \lim_{n \to \infty} \frac{1}{n^{m(k-1)}} \sum_{1 \le j_1,...,j_s\le n} \prod_{l=1}^s \big( \theta j_l^{m_l(k-1) - 1} \big) (\{j_l\alpha \} - \{j_l\beta \})^{m_l}
\end{multline}
Thus, the limiting moments will differ.

The authors in \cite{arousdang} were motivated by linear eigenvalue statistics, and therefore a specific choice of $(u_j)_{j \ge 1}$.  For each $\sigma \in \mathfrak{S}_n$, let $E_{n,k}^{\textrm{tuple}}(\sigma)$, $E_{n,k}^{\textrm{set}}(\sigma)$, and $E_{n,k}^{\textrm{irrep}}(\sigma)$ denote the multiset of eigenangles of $\rho_{n,k}^{\textrm{tuple}}$, $\rho_{n,k}^{\textrm{set}}$, and $\rho_{n,k}^{\textrm{irrep}}$ respectively.  Let $f$ be a real-valued periodic function with period 1.  Define the linear eigenvalue statistic \begin{equation} X_{n, k, f}^{\textrm{tuple}}(\sigma) = \sum_{\phi \in E_{n,k}^{\textrm{tuple}}(\sigma)} f(\phi)
\end{equation} and define $X_{n, k, f}^{\textrm{set}}(\sigma)$ and $X_{n, k, f}^{\textrm{irrep}}(\sigma)$ similarly.  Let \begin{equation} R_j(f) = \frac{1}{j} \bigg(\frac{1}{2} f(0) + \sum_{i=1}^{j-1} f\left(\frac{i}{j}\right) + \frac{1}{2}f(1) \bigg) - \int_0^1 f(x)dx
\end{equation} 
One can interpret $R_j(f)$ as the error in approximating the integral using the trapezoidal rule.  
Then it is easy to see (i.e. \cite[(1.8)]{arousdang}) that 
\begin{equation} \label{Xnkftuple} X_{n,k, f}^{\textrm{tuple}}(\sigma) = \big| E_{n,k}^{\textrm{tuple}}(\sigma) \big| \int_0^1 f(x)dx + \sum_{j} C_{j,k}^{(n), \textrm{tuple}}(\sigma) j R_j(f) 
\end{equation}
 \begin{equation} \label{Xnkfset} X_{n,k, f}^{\textrm{set}}(\sigma) = \big| E_{n,k}^{\textrm{set}}(\sigma) \big| \int_0^1 f(x)dx + \sum_{j} C_{j,k}^{(n), \textrm{set}}(\sigma) j R_j(f) 
\end{equation}

In particular, finding the limiting behavior of the linear statistic $\displaystyle{ X_{n, 1, f}^{\textrm{tuple}}(\sigma)}$ corresponds to investigating $X_n^{(u_j)}$ for $u_j = j R_j(f)$.  This is the case studied in \cite{arousdang} for a wide class of functions $f$.  For smooth functions with good trapezoidal approximations, $R_j(f)$ will decay to zero rapidly.  Thus, we have a direct correspondence between smoothness of the function $f$ and decay rate of $u_j$.  

For $k>1$, define
\begin{equation}
Y_{n,k, f}^{\textrm{tuple}} := \frac{X_{n,k, f}^{\textrm{tuple}} - \mathbb{E}[X_{n,k, f}^{\textrm{tuple}}]}{n^{k-1}}
\end{equation} 
\begin{equation}
Y_{n,k, f}^{\textrm{set}} := k! \frac{X_{n,k, f}^{\textrm{set}} - \mathbb{E}[X_{n,k, f}^{\textrm{set}}]}{n^{k-1}}
\end{equation} 
\begin{equation}
Y_{n,k, f}^{\textrm{irrep}} := k! \frac{X_{n,k, f}^{\textrm{irrep}} - \mathbb{E}[X_{n,k, f}^{\textrm{irrep}}]}{n^{k-1}}
\end{equation} 
With these definitions, we can state the following generalization of Theorem \ref{simplifythm}.

\begin{theorem} \label{simplifythmgeneral}
Let $f$ be such that $R_j(f) = O(1/j)$.  Let \begin{equation}
Y_{n,k,f} = \sum_{j=1}^n \frac{C_j^{(n)}  j^k R_j(f)}{n^{k-1}}
\end{equation} 
As $n \to \infty$ for fixed $k > 1$, each of the random variables $Y_{n,k,f}^{\textrm{tuple}}$, $Y_{n,k,f}^{\textrm{set}}$, and $Y_{n,k,f}^{\textrm{irrep}}$ converges in law to the same limiting distribution $\displaystyle{\lim_{ n\to \infty} \mathcal{L} (Y_{n,k,f})}$ (assuming it exists). 
\end{theorem}

\begin{proof}
The proof of Theorem \ref{simplifythm} goes through essentially unchanged for each of the three types of representations. 
We have equation \eqref{Xnkftuple} analogous to \eqref{Xnktuple}.  Since $j R_j(f)$ is $O(1)$, Lemma \ref{tuplebound} applies and the proof follows unchanged for the $k$-tuple case.  Similarly, we have equation \eqref{Xnkfset} analogous to \eqref{Xnkset}. Lemma \ref{subsetbound} also applies unchanged for the $k$-subset case.  Finally, for the irrep case, the generalization of Lemma \ref{Qtuplestar} replacing $\displaystyle{E_n^{Q_{n,k}^\textrm{tuple*}} \cap I}$ with $\displaystyle{\sum_{\phi \in E_n^{Q_{n,k}^\textrm{tuple*}}} f(\phi) }$ clearly holds.  Then the same induction argument on $k$ gives the appropriate generalization of Lemma \ref{Qtuple} and the result follows.  
\end{proof}

\begin{remark}
By Lemma 5.3 in \cite{arousdang}, if $f$ is of bounded total variation, then $R_j(f) = O(1/j)$. 
\end{remark}

\begin{remark}
Theorem \ref{simplifythm} corresponds to the case $f = \mathbbm{1}_{(\alpha, \beta)}$ where $\alpha$ and $\beta$ are linearly independent irrational numbers over $\mathbb{Q}$.   
\end{remark}
  
From the perspective of Theorem \ref{simplifythmgeneral}, we see that whereas \cite{arousdang} studies the random variables $X_n^{(u_j)}$ for $u_j = j R_j(f)$ for functions $f$ of various degrees of smoothness, the present work investigates them mostly for $u_j = j^k R_j(f)$ where $f = \mathbbm{1}_{(\alpha, \beta)}$.  To conclude, we observe that we can obtain limiting laws for (appropriately scaled versions of) $Y_{n,k,f}$ (and therefore $Y_{n,k,f}^{\textrm{tuple}}$, $Y_{n,k,f}^{\textrm{set}}$, and $Y_{n,k,f}^{\textrm{irrep}}$) that match those seen in Theorems \ref{smooth} and \ref{rough} by choosing $f$ so that $R_j(f)$ satisfies appropriate conditions.  For instance, one can show via Euler-Maclaurin summation that if $f \in C^{2m}$, i.e. $2m$ times continuously differentiable, then $R_j(f) = O(1/j^{2m})$.  Then if for even $k$ we take $f \in C^{k+2}$ and for odd $k$ we take $f \in C^{k+1}$, the hypotheses of Theorem \ref{smooth} are met.  If we choose $f$ on the cusp of $k$-differentiability such that $R_j(f) = \Theta(1/j^k)$, then the hypotheses of Theorem \ref{rough} are met.

\section*{Acknowledgements}
The author wishes to thank his PhD advisor Steve Evans for helpful discussions and comments on this work.  

\bibliography{references}
\bibliographystyle{plain}

\end{document}